\documentclass[reqno,11pt]{amsart}

\usepackage{amsthm,amsfonts, amssymb, amscd}
\usepackage[all, cmtip]{xy}
\usepackage{euscript}

\usepackage{tikz}
\usetikzlibrary{matrix,arrows}
\usetikzlibrary{positioning}

\usepackage[normalem]{ulem}

\usepackage{hyperref}
\hypersetup{%
    linktoc=page
}

\let\oldtocsection=\tocsection
\let\oldtocsubsection=\tocsubsection
\renewcommand{\tocsection}[2]{\hspace{0em}\oldtocsection{#1}{#2}}
\renewcommand{\tocsubsection}[2]{\hspace{1em}\oldtocsubsection{#1}{#2}}
\newtheorem{Thm}{Theorem}[section]
\newtheorem{Lem}[Thm]{Lemma}
\newtheorem{Cor}[Thm]{Corollary}
\newtheorem{Prop}[Thm]{Proposition}

\theoremstyle{remark}
\newtheorem{Rem}[Thm]{Remark}
\theoremstyle{remark}
\newtheorem{example}[Thm]{Example}
\theoremstyle{definition}

\theoremstyle{definition}

\theoremstyle{definition}

\newtheorem{Def}[Thm]{Definition}

\numberwithin{equation}{section}
\newcommand{\A}{\mathbb{ A}}
\newcommand{\G}{\mathbb{ G}}
\newcommand{\Z}{\mathbb{ Z}}          % Use for integers.
 
\newcommand{\PP}{\mathbb{ P}} 
\newcommand{\QQ}{\mathbb{ Q}}        %Use for rational numbers.
\newcommand{\F}{\mathbb{ F}}           % Use for a field
\newcommand{\Ad}{\operatorname{Ad }}             % Use for Adjoint action
\newcommand{\Hom}{\operatorname{Hom}}

\newcommand{\spec}{\operatorname{Spec}}

\newcommand{\diag}{\operatorname{diag}}

\newcommand{\res}{\operatorname{res}}

\newcommand{\lgr}{\operatorname{gr }}
\newcommand{\ev}{\operatorname{ev}}

\newcommand{\mult}{\operatorname{mult}}

\newcommand{\SL}{\operatorname{SL}}

\newcommand{\fb}{{\mathfrak b}}

\newcommand{\fg}{{\mathfrak g}}
\newcommand{\fh}{{\mathfrak h}}

\newcommand{\fl}{{\mathfrak l}}
\newcommand{\fm}{{\mathfrak m}}

\newcommand{\ft}{{\mathfrak t}}
\newcommand{\fu}{{\mathfrak u}}

\newcommand{\ga}{\alpha}
\newcommand{\gb}{\beta}

\newcommand{\gre}{\epsilon}

\newcommand{\gl}{\lambda}

\newcommand{\gz}{\zeta}

\newcommand{\ca}{\mathcal{A}}
 
\newcommand{\cc}{\mathcal{C}}

 \newcommand{\cl}{\mathcal{L}}
 
 \newcommand{\cn}{\mathcal{N}}
 \newcommand{\co}{\mathcal{O}}

  \newcommand{\cu}{\mathcal{U}} 
 \newcommand{\cv}{\mathcal{V}}
 
 \newcommand{\cx}{\mathcal{X}}

\renewcommand{\bar}[1]{\overline{#1}}

\newcounter{myenumi}
\renewcommand{\themyenumi}{$(\arabic{myenumi})$}
\newenvironment{myenumerate}{%
\setlength{\parindent}{10pt}% if 0, means don't indent paragraphs
\setcounter{myenumi}{0}% restart numbering
\renewcommand{\item}{% new definition of item
\par% start a new line
\refstepcounter{myenumi}% advance counter
\makebox[2.0em][l]{\themyenumi}% print counter to width of 3em, aligned to left
}% end of definition of item
}{% at end of environment
\par% start new paragraph
\noindent% don't indent new paragraph
\ignorespacesafterend% ignore spaces after environment
}

\begin{document}
\parskip=4pt
\baselineskip=14pt

%%%%%%%%%%%%%%%%%%%%%%%%%%%%%%%%%%%%%%%%%%%%%%%%%%%%
%%%%%%%%%%%%%%%%%%%%%%%%%%%%%%%%%%%%%%%%%%%%%%%%%%%%
%%%%%%%%%%%%%%%%%%%%%%%%%%%%%%%%\end{center}%%%%%%%%%%%%%%%%%%%%
\title[Cominuscule points and Schubert varieties]{Cominuscule points and Schubert varieties}

\author{William Graham}
\address{
Department of Mathematics,
University of Georgia,
Boyd Graduate Studies Research Center,
Athens, GA 30602
}
\email{wag@uga.edu}

\author{Victor Kreiman}
\address{
Department of Mathematics,
University of Wisconsin - Parkside,
Kenosha, WI 53140
}
\email{kreiman@uwp.edu}


\date{\today}

\begin{abstract}
We introduce the notion of a cominuscule point in a Schubert variety in a generalized flag variety
for a semisimple group.  
We derive formulas expressing the Hilbert series and 
multiplicity of a Schubert variety at a cominuscule point in terms of the restrictions of
classes in torus-equivariant K-theory and cohomology to that point, generalizing
previously known formulas for flag varieties of cominuscule type. 
Thus, we can calculate Hilbert series and multiplicities in cases
where these were previously unknown. The formulas for Schubert varieties are special cases of more general formulas valid at
generalized cominuscule points of schemes with
torus actions.
\end{abstract}

\maketitle

\tableofcontents 

\section{Introduction}
Torus-equivariant $K$-theory and Chow groups can be used to calculate Hilbert
series and multiplicities for Schubert varieties in a cominuscule flag variety.
The purpose of this paper is to extend these methods to calculate Hilbert series
and multiplicities at certain other points of Schubert varieties, which we call
cominuscule points.  Our methods apply in the more general setting of generalized cominuscule points
of varieties with torus actions.

Hilbert series and multiplicity calculations on a cominuscule flag variety
are possible because such a flag variety $X$
has an action of a torus $T$ with the following property.  For any $T$-fixed point $x$ in $X$,
there is an element $v$ in the Lie 
algebra of $T$ such that $\ga(v) = -1$ for any weight $\ga$ of $T$ on the $T$-representation $T_x X$.   Note that the
condition concerns the tangent space to the ambient variety $X$, not the Schubert variety.
See \cite{GrKr:15} and \cite{Ike:16} for a discussion and references.

The key observation of this paper is that these calculations can be
carried out under a
weaker tangent space condition, where the tangent space of the ambient flag variety $X$ is replaced
by the tangent space of a slice to the Schubert variety.  If there exists $v$ such that $\ga(v) = -1$
for all weights $\ga$ of the tangent space to the slice at $x$, we call the point $x$ a cominuscule point of the Schubert
variety.
If $X$ is a
 cominuscule flag variety, any torus-fixed point of a Schubert variety is a cominuscule point.  However,
our condition is more widely satisfied, and provides new examples where the calculation of Hilbert series and multiplicities
is possible.  
 
There is an effective test for whether a point in a Schubert variety is cominuscule, in
case $G$ is classical.
The reason is that the tangent spaces
of the slices can be described in terms of the tangent
spaces of the Schubert varieties, and for classical groups,
these have been described combinatorially.  Using this, we produce examples
of cominuscule points in Schubert varieties which are
essentially new, in the sense that the Schubert varieties are not
inverse images of Schubert varieties in cominuscule flag varieties
under a projection of
flag varieties.

In a sequel \cite{GrKr:19} to this
paper we further explore computational and combinatorial aspects of
cominuscule points.  In particular, in type $A$, we characterize these
points and give combinatorial rules for their Hilbert series and
multiplicities in terms of the pipe dreams of Fomin-Kirillov
\cite{FoKi:96}, Bergeron-Billey \cite{BeBi:93}, and Knutson-Miller
\cite{KnMi:05}.  We also plan to explore generalizations of the
excited Young diagrams of \cite{IkNa:09}, \cite{GrKr:15},
\cite{Kre:05}, and \cite{Kre:06} in this setting.  

Slices to Schubert varieties were previously studied by
Li and Yong \cite{LiYo:12}.  They work in type $A$; the slices they consider are Kazhdan-Lusztig varieties,
and they study these varieties by using coordinates to study their ideals.
They observe that if the Kazhdan-Lusztig
variety is invariant under dilation---which occurs if the Weyl group element
defining the fixed point is $\lambda$-cominuscule in the sense of Peterson (cf.~Section  \ref{ss.cominusculeweyl})---then 
their methods
can be used to obtain Hilbert series and multiplicities.  
Our approach differs in that the slices we use are generally smaller than the Kazhdan-Lusztig
varieties.  Moreover, 
instead of using the slice itself, we (in effect) replace the slice by
its tangent cone at the point, so that we can work inside
inside the tangent space at the
point. See Example
\ref{ex.subtlety} for some discussion and an example.  Thus, we can obtain
Hilbert series and multiplicity formulas in cases where the results of
\cite{LiYo:12} do not apply.  Moreover, our results are not limited to type $A$.

The contents of the paper are as follows.  Section
\ref{s.preliminaries} contains some definitions and
results about $K$-theory and intersection theory of schemes
with torus actions.
Section \ref{s.cominuscule}
defines the notion of a generalized cominuscule point of
 a scheme with a torus action.
The remainder of the paper focuses on Schubert varieties.
Section \ref{s.Schubertslice} contains some background about algebraic groups, Weyl groups, and Schubert varieties, leading to the definition of the slices to Schubert varieties
used in the definition of cominuscule points.
Section \ref{s.cominusculeSchubert} defines the notion of cominuscule points of
Schubert varieties, gives some examples, and obtains Hilbert series and multiplicity formulas in this
case.  Some of the results can be stated more simply in type $A$, because the tangent
spaces to Schubert varieties are easier to describe.
Section \ref{s.examples} contains some additional examples.

\section{Notation and preliminaries} \label{s.preliminaries}
\subsection{Cones}  \label{ss.cones}
We work with schemes of finite type over an algebraically closed
field $\F$.  
Given a (finite-dimensional) vector space $V$ over $\F$, 
$S(V)$ denotes the symmetric algebra of $V$, that is, the
ring of polynomials on $V^*$.  
Given a scheme $X$,
$A_i(X)$ denotes the $i$-th Chow group of $X$, and
$A^*(X)$ denotes the operational Chow ring of $X$ (see \cite{Ful:84}).
We write
$h = c_1(\co_{\PP(V)}(1)) \in A^1(\PP(V))$.   A cone $C$ is a subscheme
of $V$ which is invariant under the action of the multiplicative
group $\G_m$.  If $C$ is a closed cone in $V$ of pure dimension,
and $k$ is the codimension of $C$ in $V$,
then $[\PP(C)] = a h^k [\PP(V)] \in A_*(\PP(V))$;
the degree of the cone is the integer
$a$.

\subsection{The tangent cone, multiplicities and Hilbert series} \label{ss.tangentcone}
Let $x$ be a closed point in a scheme $X$.  Working locally, we may assume $X = \spec A$.  Let $\fm$
be the maximal ideal corresponding to $x$.  The tangent
space (by which we mean the Zariski tangent space)
 is $T_x X = \fm/\fm^2$.  The tangent cone to $X$ at $x$ is defined as
$C_x X = \spec (\mbox{gr}_{\fm} A)$, where $\mbox{gr}_{\fm} A =
\oplus_i \fm^i / \fm^{i+1}$.  There is a surjection
$S(\fm/\fm^2) \to \mbox{gr}_{\fm} A$, so
$C_x X$ is a closed cone in $T_x X$, and the multiplicity
$\mult(X,x)$ is the degree of this cone.

The Hilbert function is the function $n \mapsto \dim (\fm^n/\fm^{n+1})$.
For sufficiently large values of $n$ it is a polynomial $h(X,x)(n)$
in $n$, called the Hilbert polynomial; this is related to the multiplicity
by the equation $\mult(X,x) = a_d/d!$, where
$a_d$ is the leading coefficient of $h(X,x)(n)$.
The Hilbert series of $X$ at $x$ 
is the
power series $H(X,x) = \sum \dim(\fm^i/\fm^{i+1}) t^i$.
Observe that $H(X,x) = H(C_x X,x)$ and
$\mult(X,x) = \mult(C_x X,x)$.  If $x$ and $y$ are closed points in $X$ and $Y$, respectively,
then $C_{(x,y)}(X \times Y) \cong C_x X \times C_y Y$.  Hence
\begin{equation} \label{e.productcone}
H(X\times Y,z)=H(X,x)H(Y,y) \mbox{   and   } \mult (X \times Y, z) = \mult (X,x) \mult(Y, y).
\end{equation}

\subsection{Tori, completions, and evaluation maps} \label{ss.torus}
Let $T \cong (\G_m)^n$ be a torus.  The character group of $T$ is $\hat{T} = \Hom(T, \G_m)$.  We can view
$\hat{T}$ as a subset of the dual $\ft^*$ of the Lie algebra of $T$.  If we want to view $\gl \in \hat{T} \subset \ft^*$
as a homomorphism $T \to \G_m$, we will
write it as $e^{\gl}$.  We write 
$\F_{\gl}$ for the  
$1$-dimensional representation of $T$ of weight $\gl$
(that is, on which $T$ acts by $e^{\gl}$).
If $V$ is a representation of $T$, we denote by $\Phi(V) \subset \hat{T}$
the set of weights of $T$ on $V$.  
We identify $\Hom(\G_m, T)$ with the set of $v \in \ft$ such that $\gl(v) \in \Z$ for all $\gl \in \hat{T}$; we will
say that such a $v$ is integral.  If a nonzero integer multiple of $v$ is integral, 
we say $v$ is rational.  

The representation ring $R(T)$ of $T$ 
is the free $\Z$-module with basis $e^{\gl}$, for
$\gl \in \hat{T}$, and multiplication given by
$e^{\gl}e^{\mu} = e^{\gl+\mu}$.  Let $S(\hat{T})$ be the symmetric algebra
on $\hat{T}$; if $\gl_1, \ldots, \gl_n$ is a basis for 
$\hat{T} \cong \Z^n$, then $S(\hat{T})$ is the polynomial ring
$\Z[\gl_1, \ldots, \gl_n]$.  
  
Let $v \in \ft$ be rational, and let $d$ be a positive integer
such that $d v$ is integral.  For $i \in \frac{1}{d} \Z$,
let $R^i(T)$ be the span of $e^{\gl}$ with $\gl(v) = i$.
Then 
$R(T) = \oplus_i R^i(T)$.  
Let $\hat{R}(T)$ be the
completion of $R(T)$ with respect to the ideal of positive
degree elements.  An element of $\hat{R}(T)$ can be
written as a (possibly) infinite sum $\sum r_i$,
where $r_i \in R^i(T)$, and the set of $i$ such that 
$r_i \neq 0$ is bounded below. If $v$ is integral,
define a homomorphism
$\ev_v: R(T) \to \Z[t,t^{-1}]$ by $\ev_v(e^{\gl}) = t^{\gl(v)}$.
We extend this to
$\ev_v: \hat{R}(T) \to \Z[t^{-1}][[t]]$ by
$\ev_v(\sum r_i) = \sum \ev_v(r_i)$.  If $v$ is rational, this construction gives a map
$\ev_v: \hat{R}(T) \to \Z[u^{-1}][[u]]$, where
$u = t^{1/d}$.  Note that
if $f,g \in R(T)$ such that $g$ is a unit in $\hat{R}(T)$,
then $\ev_v(f/g)$ is the expansion of the
rational function $\ev_v(f)/\ev_v(g)$ in positive
powers of $t$ (i.e., the Laurent series expansion at $t = 0$,
or $u = 0$ in the rational case).

We can also define an evaluation map using $S(\hat{T})$ in
place of $R(T)$.  Precisely, if $\gl_1, \ldots, \gl_n$ is
a basis of $\hat{T}$, then $S(\hat{T})$ is
the polynomial ring $\Z[\gl_1, \ldots, \gl_n]$.  If
$f(\gl_1, \ldots, \gl_n)$ and $g(\gl_1, \ldots, \gl_n)$ are
in $S(\hat{T})$, define
$$
\ev_v \Big( \frac{f}{g} \Big) = \frac{f(\gl_1(v), \ldots, \gl_n(v))}{g(\gl_1(v), \ldots, \gl_n(v))} \in \QQ,
$$
provided the denominator is nonzero.

\subsection{Equivariant $K$-theory and Chow groups for torus actions} \label{ss.equivariant}
If a torus $T$ acts on a scheme $X$, we denote by $K_T(X)$ (resp.~$G_T(X)$) the
Grothendieck group of $T$-equivariant vector bundles (resp.~coherent
sheaves) on $X$.  If $X$ is nonsingular, the natural map $K_T(X) \to
G_T(X)$ is an isomorphism (see \cite[Cor.~7.8]{Tho:87}).  The representation ring
$R(T)$ is identified with the Grothendieck group $K_T(pt)$ of a point, and $K_T(X)$ is an $R(T)$-module.  Similarly, we let $A^T_*(X)$ denote the
$T$-equivariant Chow groups of $X$ (see
\cite{EdGr:98}).  By definition, $A_i^T(X) = A_{i
  +N}(\cu \times^T X)$, where $\cu$ is an open $T$-invariant subset in
a representation $\cv$ of $T$ such that the codimension of $\cv
\setminus \cu$ in $\cv$ is greater than $\dim X - i$, and
$N = \dim \cv - \dim T$.  
We abuse notation and write $X_T = \cu \times^T X$,
although $X_T$ depends on the choice of a suitable $\cu \subset \cv$.
The ring $S(\hat{T})$
can be identified with the operational Chow ring $A^*_T(pt)$ of a
point.  This acts on $A_*^T(X)$; by definition, the element $\gl \in
\hat{T}$ acts by multiplication by the first Chern class of the line
bundle $\cu \times^T (X \times \F_{\gl}) \to \cu \times^T X$.  If $T$
acts freely on $X$, then flat pullback via the map $q: \cu \times^T X
\to X/T$ induces an isomorphism $A_i(X/T) \to A_i^T(X)$.  

A $T$-equivariant vector bundle $V$ on $X$ has equivariant Chern classes
$c_i^T(V)$ in the operational equivariant Chow groups 
$A^i_T(X)$.  A representation $V$ of $T$ can be viewed as a $T$-equivariant vector bundle
over a point.  If $\dim V = d$, then the top equivariant Chern class of
$V$ is $c_d^T(V) = \prod_{\ga \in \Phi(V)} \ga \in A^T_*(pt)$.
Another important class
is $\gl_{-1}(V^*) = \prod_{\ga \in \Phi(V)}  (1 - e^{-\ga}) \in K_T(pt) = R(T)$.

For any point $x$ with a trivial $T$-action
we identify $K_T(\{ x \})$ with $R(T)$ and $A^*_T(\{ x \})$ with
$S(\hat{T})$.  Thus, if $i_x: \{ x \} \to X$ is the inclusion
of a $T$-fixed point of a nonsingular $T$-variety, we have
pullback maps $i_x^*: K_T(X) \to K_T(\{ x \}) = R(T)$ and 
$i^*: A^T_*(X) \to A^*_T( \{ x \}) = S(\hat{T})$.  If $i_V: \{ 0 \} \to V$
is the inclusion of $\{ 0 \}$ into a representation $V$, then pullback induces isomorphisms
$i_V^*: K_T(V) \to R(T)$ and $i_V^*:A^*(V) \to S(\hat{T})$.

Let $X$ be a $T$-scheme (i.e., a scheme with $T$-action), and $Y$ a closed $T$-stable subscheme.
We denote by $[\co_Y]_X$ and $[Y]_X$ the structure sheaf and equivariant fundamental classes of $Y$ in
$G_T(X)$ and $A_*^T(X)$, respectively.
If it is understood that we are working in 
$G_T(X)$ or $A_*^T(X)$, the subscript $X$ is frequently omitted.  

\begin{Lem} \label{lem.restriction}
Suppose a torus $T$ acts on smooth varieties $M_1$, $M_2$,
and suppose $Y$ is a closed subscheme of $M_2$.  Let
$m_i$ be a $T$-fixed point of $M_i$ ($i=1,2$), and let
$i_{m_1}$, $i_{m_2}$, and $i_{(m_1, m_2)}$ be the inclusions of $m_1$, $m_2$ and $(m_1, m_2)$ into $M_1$, $M_2$ and
$M_1 \times M_2$, respectively.
Then
$$
i_{(m_1, m_2)}^*[\co_{M_1 \times Y}]_{M_1 \times M_2}
= i_{m_2}^* [\co_Y]_{M_2}
$$
and
$$
i_{(m_1, m_2)}^*[M_1 \times Y]_{M_1 \times M_2}
= i_{m_2}^* [Y]_{M_2}.
$$
\end{Lem}

\begin{proof}
Let $k: \{ (m_1,m_2) \} \to \{ m_2 \}$.  Under our identifications 
of $K_T( \{ (m_1,m_2) \} )$ and $K_T( \{ m_2 \} )$ with $R(T)$,
the $K$-theory pullback $k^*$ is the identity.  Similar remarks
apply for equivariant Chow groups.

Let $\pi: M_1 \times M_2$ be projection on the second
factor.  Then 
$$
i_{(m_1,m_2)}^* \pi^*
= (\pi \circ i_{(m_1,m_2)})^* = (i_{m_2} \circ k)^*
= k^* i_{m_2}^* = i_{m_2}^*.
$$
Also, $\pi^*([\co_Y]_{M_2}) = [\co_{M_1 \times Y}]_{M_1 \times M_2}$ 
and $\pi^*([Y]_{M_2}) = [M_1 \times Y]_{M_1 \times M_2}$.
Hence,
$$
i_{(m_1, m_2)}^*[\co_{M_1 \times Y}]_{M_1 \times M_2}
= i_{(m_1, m_2)}^* \pi^*([\co_Y]_{M_2}) = i_{m_2}^*([\co_Y]_{M_2})
$$
and
$$
i_{(m_1, m_2)}^*[M_1 \times Y]_{M_1 \times M_2}
= i_{(m_1, m_2)}^* \pi^*([Y]_{M_2}) = i_{m_2}^*([Y]_{M_2}),
$$
as desired.
\end{proof}

\begin{Lem} \label{lem.zerosection}
If $\cl$ is a $T$-equivariant line bundle on 
a $T$-scheme $X$, and $s$ is a $T$-invariant regular section of $\cl$ (cf.~\cite[Section 14.1]{Ful:84}) with 
zero-scheme $Y$, then $[Y] = c_1^T(\cl) \cap [X]$ in $A_*^T(X)$, and $[\co_Y] = [\co_X] -[\cl^*]$
in $G_T(X)$.
\end{Lem}

\begin{proof}
By definition, $A_i^T(X) = A_{i+ N} (X_T)$, where $X_T$ and
$N$ are
as above.
The $T$-equivariant line bundle $\cl$ defines a line bundle $\cl_T$ on $X_T$ whose
zero-scheme is $[Y_T]$.  By \cite[Prop.~14.1]{Ful:84}, $[Y_T] = c_1(\cl_T) \cap [X_T]$ in
$A_*(X_T)$.  Since by definition the class $[Y]$ in $A_*^T(X)$ is the class $[Y_T]$
in $A_*(X_T)$, and $c_1^T(\cl)$ is $c_1(\cl_T) \in A^1(X_T)$, we see that
$[Y] = c_1^T(\cl) \cap [X]$ in $A_*^T(X)$.  For the assertion about classes in $G_T(X)$, see
\cite[Remark 6.3]{GrKu:08}.
\end{proof}

\begin{Lem} \label{lem.equivariant}
Let $T \cong \G_m$ and let $\gl$ be a generator of $\hat{T}$.
Assume $T$ acts on $V = \F^{n+1}$ with all weights equal to
$r \gl$, $r \neq 0$.  Let $V^0 = V \setminus \{ 0 \}$.  

\begin{myenumerate}
\item Suppose $r = -1$, so $T$ acts freely on $V^0$.  Under the isomorphism 
$A_*^T(V^0) \cong A_*(\PP(V))$, the action of $\gl \in A^1_T(pt)$ on
$A_*^T(V^0)$
corresponds to multiplication by $- h$ on $A_*(\PP(V))$, where
$h = c_1(\co_{\PP(V)}(1))$.  

\item If $C \subset V$ is a closed cone of pure
codimension $k$, then for some $a \in \Z$,
$[\PP(C)] =   a h^k \cap [\PP(V)]$ in $A_*(\PP(V))$ and
$[C] =  a  r^k \gl^k \cap [V]$ in $A_*^T(V)$.
\end{myenumerate}
\end{Lem}

\begin{proof} 
(1) The element $\gl \in A^1_T(pt)$ acts 
on $A_*^T(V^0)$, so via the isomorphism 
$A_*^T(V^0) \cong A_*(\PP(V))$, $\gl$ acts 
on $A_*(\PP(V))$.  The action of $\gl$
on $A_*(\PP(V))$ is multiplication by the first Chern class
of the line bundle $(V^0 \times \F_{\gl})/T \to V^0/T = \PP(V)$.
This line bundle is isomorphic to the
tautological subbundle $S$ of the trivial bundle
$V \times \PP(V) \to \PP(V)$.  Part (1) follows since $c_1(S) = -h$.  

(2) First assume $r = -1$.  We have $[\PP(C)] = a h^k [\PP(V)]$ for some $a \in \Z$.
Write $C^0 = C \setminus \{ 0 \}$.
Under the isomorphism of
$A_*^T(V^0)$ with $A_*(\PP(V))$, the class $[C^0]$ corresponds
to $[\PP(C)]$, and we have proved that the action of $\gl$ corresponds
to the action of $-h$. Hence, $[C^0] = (-1)^k a \gl^k [V^0]$
in $A^T_*(V^0)$.  
The restriction map 
$A_{n+1-k}^T(V) \to A_{n+1-k}^T(V^0)$ (where $n+1 = \dim V$) is an isomorphism taking
$[C]$ to $[C^0]$ and $\gl^k[V]$ to $\gl^k[V^0]$, so $[C] = (-1)^k a \gl^k [V]$.
This proves (2) for $r = -1$.

We now consider the case of general $r$. Define
$\phi: T \cong \G_m \to T_1 = \G_m$ by $t \mapsto t^{-r}$.  Let
$\gl_1 \in \hat{T_1}$ be the character whose image under the
pullback $\phi^*: \hat{T_1} \to \hat{T}$ equals $-r \gl$.
If $T_1$ acts on $V$ with all weights equal to $\gl_1$, then 
the $T$-action on $V$ is induced by the $T_1$-action via the map $\phi$.  
There is a pullback map $\phi^*: A^*_{T_1}(V) \to A^*_T(V)$,
defined as follows.  Let $\cu$ and $\cu_1$ be open subsets of representations of $T$ and $T_1$
respectively, as in the definition of equivariant Chow groups given in Section \ref{ss.equivariant}.
The torus $T$ acts on $\cu_1$ via the map $\phi$.  Thus, there is a map
$$
(\cu \times \cu_1) \times^T V \to \cu_1 \times^{T_1} V,
$$
and pullback along this map yields $\phi^*$.   
By part (1), $[C]_{T_1} =  (-1)^k a \gl^k [V]_{T_1}$.  Also, $\phi^*(\gl_1) = -r \gl$, and
$\phi^* [C]_{T_1} = [C]_T$ , $\phi^* [V]_{T_1} = [V]_T$, where
the subscripts denote which equivariant group we are considering.
Part (2) follows.
\end{proof}

\section{Generalized cominuscule points} \label{s.cominuscule}
In this section we define the notion of a generalized cominuscule point of
 a scheme with a torus action.  We give formulas for
the Hilbert series and multiplicity at a generalized
cominuscule point (see Theorem \ref{thm.hilbert}), generalizing formulas
used by Ikeda-Naruse and Graham-Kreiman; see \cite{GrKr:15} for a discussion and references.

\subsection{Degenerating to the tangent cone} \label{ss.degenerating}
Let $V$ be a representation of a torus $T$.  Let $V^*$ denote the dual representation of $T$, so the symmetric algebra
$S(V^*)$ is the ring of regular functions on $V$.  Let $X = \spec A$ be a $T$-invariant closed subscheme of $V$ containing $0$,
so there is a surjection $S(V^*) \to A$.  Let $\fm \subset A$ be the maximal ideal of $0$.  The tangent cone of $X$ at $0$
is $\cc = \spec B$, where $B = \lgr_{\fm} A = A/\fm \oplus \fm / \fm^2 \oplus \cdots$.  There is a $T$-equivariant surjection $S(V^*) \to B$,
so $\cc$ is a $T$-invariant closed subscheme of $V$.
The next proposition is known, but for lack of a precise reference we provide a proof.  

\begin{Prop} \label{p.class}
With the assumptions and notation of the previous paragraph, we have:

$(1)$ If $X$ has pure dimension
$k$, then so does $\cc$, and
$[\cc] = [X]$ in $A^T_k(V)$.

$(2)$  $[\co_{\cc}] = [\co_X]$ in $G_T(V)$.
\end{Prop}

\begin{proof}
Let $\ca = \cdots \oplus u^{-2} \fm^2 \oplus u^{-1} \fm \oplus A \oplus uA \oplus u^2 A \oplus \cdots$ denote the Rees algebra of $A$ (see
\cite[Section 6.5]{Eis:95}).
We extend the $T$-action on $A$ to an action on $\ca$ by requiring $t \cdot u^k a = u^k (t \cdot a)$ for $k \in \Z$, $a \in A$.
Let $\cx = \spec \ca$.  There is a $T$-equivariant $\F[u]$-algebra homomorphism
$\varphi: S(V^*)[u] \to \ca$ characterized by $\varphi(\zeta) = u^{-1} \zeta$ for $\zeta \in V^*$.  The algebra $\ca$ is spanned
by elements of the form $u^{-m} \gz_1 \cdots \gz_n = \varphi( u^{n-m} \gz_1 \cdots \gz_n)$, where $\gz_i \in V^*$ and $m \leq n$,
so $\varphi$ is surjective, and hence induces a $T$-equivariant closed embedding $\cx \subset V \times \A^1$.
Let $\cx_c$ denote the fiber of $\cx$ over $c \in \A^1$.
The composition $\cx \to V \times \A^1 \to V$ is $T$-equivariant, and it takes $\cx_1$ (resp.~$\cx_0$) isomorphically
onto $X$ (resp.~$\cc$).  The fact that if $X$ has pure dimension
$k$, then so does $\cc$, is a special case of
\cite[Appendix B.6.6]{Ful:84}.

We have a $T$-equivariant open embedding $V \times \A^1 \subset V \times \PP^1$ (where $T$ acts trivially on $\PP^1$).
Let $\overline{\cx}$ denote the closure of $\cx$ in $V \times \PP^1$
Then $\overline{\cx} \cap (V \times \A^1)$ equals 
$\cx$ as a scheme, so
under the map $\pi:\overline{\cx} \to \PP^1$, the inverse image of $\A^1$ is $\cx$.
Thus, if $\overline{\cx}_c$ denotes the fiber of $\overline{\cx}$ over $c \in \PP^1$, then if 
$c \in \A^1$, we have $\cx_c = \overline{\cx}_c$.  

There are sections $s_0$ and $s_1$ of $\co_{\PP^1}(1)$ whose zero-schemes are the points $0$ and $1$, respectively.
These sections are $T$-invariant since $T$ acts trivially on $\PP^1$.
The pullbacks $\pi^* s_0$ and $\pi^* s_1$ are $T$-invariant regular  sections of $\pi^* \co_{\PP^1}(1)$ whose zero-schemes are
$\cx_0$ and $\cx_1$, respectively.  Lemma \ref{lem.zerosection} implies that in $A^T_k( \overline{\cx})$,  $[\cx_0]$ and $[\cx_1]$ are each equal to 
$c_1^T(\pi^* \co_{\PP^1}(1)) \cap [\overline{\cx}]$, which
implies $[\cx_0] = [\cx_1]$.
Let $p$ denote the composition $\overline{\cx} \to V \times \PP^1 \to V$.  Then $p$ is proper, and it takes $\cx_0$ isomorphically
onto $\cc$, and $\cx_1$ isomorphically onto $X$.
Thus, $
[\cc] = p_*[\cx_0] = p_*[\cx_1] = [X]$, 
proving (1).
Similarly, in $G_T(\overline{\cx})$, we have $[\co_{\cx_0}] = [\co_{\cx_1}]$, since by Lemma
\ref{lem.zerosection}, each is equal to
$[\co_{\overline{\cx}}] -  [\pi^* \co_{\PP^1}(-1)]$.  Thus,
$[\co_{\cc}] = p_*[\co_{\cx_0}] = p_*[\co_{\cx_1}] = [\co_X]$,
proving (2).
\end{proof}

Recall that $i_V$ denotes the inclusion of $\{ 0 \}$ into a representation
$V$. 

\begin{Prop} \label{prop.cone}
Keep the notation and hypotheses of Proposition \ref{p.class}.   Suppose  $k:V_1 \hookrightarrow V$
is the inclusion of a $T$-invariant subspace such that all the weights of $T$ on
$V/V_1$ are nonzero, and suppose that $V_1$ contains $\cc$.  Let
$d = \dim V$ and $d_1 = \dim V_1$.

$(1)$ $i_{V_1}^* ([\co_{\cc}]_{V_1}) = \frac{i_V^* ([\co_{\cc}]_{V})}{\gl_{-1}((V/V_1)^*)} = \frac{i_V^* ([\co_X]_{V})}{\gl_{-1}((V/V_1)^*)}$.

$(2)$ $i_{V_1}^*([\cc]_{V_1}) = \frac{i_V^*([\cc]_V)}{c_{d-d_1}^T(V/V_1)} = \frac{i_V^*([X]_V)}{c_{d-d_1}^T(V/V_1)}$.
\end{Prop}

\begin{proof}
The self-intersection formula in equivariant $K$-theory (see e.g.~\cite[Section 3.1]{EdGr:05})
implies that $k^* k_*$ is multiplication by $\gl_{-1}((V/V_1)^*) $, since the normal bundle of $V_1$ in $V$ is the bundle
$V_1 \times (V/V_1) \to V_1$, which is trivial, but
not equivariantly trivial.  Hence
\begin{eqnarray*}
i_V^* ([\co_{\cc}]_{V}) & = &  i_V^* k_*([\co_{\cc}]_{V_1}) = i_{V_1}^* k^* k_* ([\co_{\cc}]_{V_1})\\
& =  & i_{V_1}^* ( \gl_{-1}((V/V_1)^*) [\co_{\cc}]_{V_1}) =  
\gl_{-1}((V/V_1)^*) i_{V_1}^* ([\co_{\cc}]_{V_1}).
\end{eqnarray*}
The last equality holds because $i_{V_1}^*$ is an $R(T)$-module map, so
$i_{V_1}^*$ commutes with multiplication by $\gl_{-1}((V/V_1)^*) $.
Dividing by $ \gl_{-1}((V/V_1)^*)$, and using the equality $[\co_{\cc}]_{V} = [\co_X]_{V}$, proves (1).
The proof of (2) is similar, using the fact that in equivariant Chow groups, the self-intersection
formula implies that $k^* k_*$ is multiplication by $c_{d-d_1}^T(V/V_1)$.
\end{proof}

In the above proposition, the equalities in (1) and (2) are to be interpreted as
equalities in $R(T)$ and $A_*^T(pt)$, respectively.  There is no
need to localize (i.e.~invert elements): (1) implies that $i^*[\co_X]$ is divisible
by $\gl_{-1}((V/V_1)^*)$, and (2) implies that
$i^*[X]$ is divisible by $c_{d-d_1}^T (V/V_1)$.

\subsection{Generalized cominuscule points of schemes}

\begin{Def} \label{def.generalcominuscule}
Suppose $X$ is a closed $T$-invariant subscheme of a nonsingular
$T$-variety $M$.  A point $x \in X^T$ is said to be a generalized
cominuscule point of $X$ if there are 
\begin{myenumerate}
\item Representations $V'$ and $V$ of $T$, such that all the
weights of $T$ on $V$ are nonzero, and an isomorphism of
$V' \times V$ with an open subscheme $M_0$ of $M$ containing $x$.  Let $X_0 = X \cap M_0$.  

\item A $T$-invariant subscheme $\cn$ of $V$, called a slice, such
that the isomorphism of (1) restricts to an isomorphism
of $V' \times \cn$ with $X_0$.  We identify $\cn$ with 
$\{ 0 \} \times \cn$ and via $V' \times \cn \cong X_0$
we view $\cn$ as a subscheme
of $X_0$, and $x$ as a point of $\cn$.

\item A element $v \in \ft$ (which can be assumed to be rational) such that
for each weight $\ga$ of $T$ on the 
Zariski tangent space $T_x \cn$, we have $\ga(v) = -1$.   
\end{myenumerate}
  \end{Def}
  
\begin{Rem} \label{rem.cone}
This situation differs from the situation considered in \cite[Section 9]{IkNa:09},
because $\cn$ need not be embedded as a cone in $V$.  One can find a possibly different embedding
of $\cn$ into $V$ which factors through the inclusion $V_1 = T_x \cn \hookrightarrow V$, such
that under this embedding, $\cn$ is a cone in $V_1$.  See Example \ref{ex.subtlety} below.
If we let $\cn_1$ denote the image of $\cn$ in $V$ under the 
original embedding, and $\cn_2$ the the image under the new embedding, then we obtain
classes $[\co_{\cn_1}]$ and $[\co_{\cn_2}]$ in $K_T(V)$, and $[\cn_1]$ and $[\cn_2]$ in $A^T_*(V)$.
If one showed
that $[\co_{\cn_1}] = [\co_{\cn_2}]$ and $[\cn_1] =[\cn_2]$, then some of our results and proofs (for example, part of the proof
of Theorem \ref{thm.hilbert}) would follow by applying the results described
in \cite{IkNa:09} to the classes $[\co_{\cn_2}]$ and $[\cn_2]$.  The approach in this paper is somewhat different: it uses the tangent
cone $\mathcal{C}$ to $\cn$ at $x$, which lies in $V_1$ by construction, and which
by Proposition \ref{p.class} yields the same classes in $K$-theory or Chow groups as $\cn$.
\end{Rem}

\begin{Rem} \label{rem.rational}
We can assume that the element $v$ is rational because if $S$ is a finite subset of $\hat{T}$, and 
$v \in \ft$ satisfies $\ga(v)  \in \QQ$ for all $\ga \in S$,
then there exists a rational element $v'$ of $\ft$ such that
$\ga(v') = \ga(v)$ for all $\ga \in S$.  This can be seen as follows.  Let $\{ \gl_1, \ldots, \gl_n \}$ be a
basis for $\hat{T} \otimes \QQ$ such that $\gl_1, \ldots, \gl_k$ are elements of $S$ which form a basis
for the subspace of $\hat{T} \otimes \QQ$ spanned by $S$.  
Let $v_1, \ldots, v_n$ denote the dual basis of $\Hom(\hat{T} \otimes \QQ, \QQ)$.  Let $a_i = \gl_i(v)$, and let
$v' = a_1 v_1 + \cdots + a_k v_k  \in \Hom(\hat{T} \otimes \QQ, \QQ)$.
Since $ \Hom(\hat{T} \otimes \QQ, \QQ)$ can be viewed as the set of rational elements of $\ft$, this
suffices.  Thus, if in (3), the element $v$ was not rational, we could replace it by the rational element $v'$, and
the conditions of the definition would be satisfied.
\end{Rem}

\begin{Rem}
If  $X$ is a closed $T$-invariant subscheme of a nonsingular
$T$-variety $M$, $x \in X^T$, and there exists $v \in \ft$ such that $\ga(v) = -1$ for all
$\ga \in \Phi(T_x M)$, then $x$ is a generalized cominuscule
point of $X$.  Indeed, \cite[Prop.~A2]{Bri:99} implies that there is a neighborhood $X_0$
of $x$ in $M$ isomorphic to $V = T_x M$, so taking
$V' = 0$ and $\cn = X_0$, the hypotheses of Definition
\ref{def.generalcominuscule} are satisfied.  This is the situation if $X$ is a Schubert
variety in a cominuscule flag variety $M$.
\end{Rem}

\begin{example} \label{ex.subtlety} 
It is natural to ask why
  the representation $V$ is part of the definition of cominuscule point, since it might appear that we can
  simply replace $V$ by its subspace $V_1 = T_x \cn$.  (By  \cite[Prop.~A2]{Bri:99}, under the
  hypotheses of Definition \ref{def.generalcominuscule}, there is a $T$-invariant neighborhood $\cn_0$ of
  $x$ in $\cn$ and a $T$-equivariant embedding of $\cn_0$ into $V_1$
  taking $x$ to $0$.)   However, in applications, we may have a natural embedding $\cn \subset V$
  such that $\cn_0$ does not lie in the subspace $V_1$ of $V$.  Thus, the composition $\cn_0 \to V_1 \to V$ is not the original embedding
  of
  $\cn_0$ into $V$ (cf.~Remark \ref{rem.cone}).  Here is an example (where $\cn = \cn_0$).
Let $T = \G_m$ act on $V = \A^2$ so that $v = (1,0)$ and $w = (0,1)$ are weight
vectors of weights $1$ and $2$, respectively.  Let
$\{x,y\}$ be the basis of $V^*$ dual to the basis $\{v,w\}$, 
so $x$ and $y$ are $T$-weight vectors of weights $-1$, $-2$.
Let $A = S(V^*) = \F[x,y]$.  Let $\cn$ be the subvariety
of $\A^2$ defined by the equation $x^2 = y$, so
$\cn = \spec B$ with $B = \F[x,y]/\langle x^2 - y \rangle$.  Note that $\cn$ is not
a cone (i.e. is not dilation-invariant).
Let $\bar{x}, \bar{y}$ denote the images of $x,y$ in $B$.
The surjection $V^* \to \fm/\fm^2 = V_1^*$ takes $x$ to a basis
element and $y$ to $0$.  The dual map $V_1 \to V$ identifies
$V_1$ with its image $\F \cdot v$, so $\cn$ is not contained in $V_1$.   However, the
tangent cone $\cc$ to $\cn$ at $0$ is contained in $V_1$---in fact, $\cc = V_1$.  Let $W = \F \cdot \bar{x}$,
so $\fm = W \oplus \fm^2$ and $W \cong \fm/\fm^2$ as representations of $T$.  Since $W$ generates $B$ as a ring,
there is a surjection $S(W) \to B$,
and this yields the embedding $\cn \to W^* \cong V_1$ constructed
by Brion.  Thus, we obtain
$$
\begin{array}{ccccc}
\cn & \to & V_1 & \to & V \\
(a,a^2) & \mapsto & a v & \mapsto & (a,0).
\end{array}
$$
The composition is not the original embedding of $\cn$ into $V$, because it does not take the point $(a,a^2)$ to itself.  Note that in this example, $0$ is a generalized cominuscule point of
$\cn$.

\end{example}

\begin{Rem} We can generalize Definition \ref{def.generalcominuscule}
in several ways.  For example,
we could
replace 
the point $0$ in $V'$ with a $T$-fixed point in a smooth
$T$-variety (with appropriate changes to the other conditions
of the definition).  This might be useful in considering
slices with respect to actions of groups which are not
unipotent.
However, for our application to Schubert varieties,
we do not need this generality.
\end{Rem}

The next result, Theorem \ref{thm.hilbert}, generalizes \cite[Prop.~9.1]{IkNa:09} 
(cf.~\cite[Cor.~2.11]{GrKr:15}).
The motivation for a statement of this form---in which
the slice $\cn$ does not explicitly appear---is that for Schubert varieties, the pullbacks $i_x^*[X]$ and
$i_x^*[\co_X]$ are pullbacks of Schubert classes,
which can be calculated.  

\begin{Thm} \label{thm.hilbert}
With notation as in Definition \ref{def.generalcominuscule},
suppose that $x$ is a generalized cominuscule point of $X$.  Let $d' = \dim V'$
and $d = \dim V$.  Let $i_x: \{ x \} \to M$ be the inclusion, and let $[\co_X]$ and $[X]$ denote
classes in $K_T(M)$ and $A^T_*(M)$, respectively.
The Hilbert series $H(X,x)$ is given by
\begin{equation} \label{e.thmhilbert1}
H(X,x) = \frac{1}{(1-t)^{d'}} \ev_v 
\Big( \frac{i_{x}^*[\co_{X}]}{\gl_{-1}(V^*)} \Big).
\end{equation}
The multiplicity $\mult(X,x)$ is given by
\begin{equation} \label{e.thmhilbert2}
\mult (X,x) = 
\ev_{-v} \Big( \frac{i_{x}^*[X]}{c_d^T(V)} \Big).
\end{equation}
\end{Thm}

\begin{proof}
Since the
Hilbert series and tangent cone are defined locally, as are the pullbacks to the
$K$-theory and Chow groups of $\{ x \}$,
we may assume that $M = M_0 = V' \times V$ and $X = X_0 = V' \times \cn$.
Then $i_x$ is the inclusion $i_{V' \times V}$ of $0$ into $V' \times V$,
$[\co_X]$ is $[\co_X]_{V' \times V}$, and $[X]$ is $[X]_{V' \times V}$.
By Lemma \ref{lem.restriction}, $ i_{V' \times V}^*([\co_{X}]_{V' \times V}) = i_V^*([\co_{\cn}]_V)$
and $ i_{V' \times V}^*([X]_{V' \times V}) = i_V^*([\cn]_V)$.

Let $V_1 = T_x \cn \subset V$, and let $\cc$ denote the tangent
cone to $\cn$ at $x$. By \eqref{e.productcone},
\begin{equation} \label{e.cone0}
H(X,x) = H(V',\{ 0 \}) H (\cn,x) = \frac{1}{(1-t)^{d'}} H(\cc, x),
\end{equation} 
since $H(V',\{ 0 \}) = 1/(1-t)^{d'}$, and
as observed in Section \ref{ss.tangentcone},
$H (\cn,x) = H (\cc,x)$.  
We have
\begin{equation} \label{e.hilbertcone1}
H(\cc,x) = \frac{\ev_v(i_{V_1}^* ([\co_{\cc}]_{V_1}))}{(1-t)^{d_1}}
= \ev_v \Big( \frac{i_{V_1}^* ([\co_{\cc}]_{V_1})}{\gl_{-1}(V_1^*)} \Big),
\end{equation}
where the first equality follows from Proposition 2.2 of \cite{GrKr:15}, and the second equality holds because
each $\ga \in \Phi(V_1)$ satisfies $\ga(-v) = 1$.
We have
\begin{equation} \label{e.hilbertcone2}
 \frac{i_{V_1}^* ([\co_{\cc}]_{V_1})}{\gl_{-1}(V_1^*)} =  \frac{ i_V^*([\co_{\cc}]_V) }{\gl_{-1}(V_1^*)
    \gl_{-1}((V/V_1)^*)}  = \frac{i_V^*([\co_{\cn}]_V)}{\gl_{-1}(V^*) } = \frac{i_{V' \times
      V}^*([\co_{X}]_{V' \times V})}{\gl_{-1}(V^*) } .
\end{equation}
where the first equality is by Proposition \ref{prop.cone} and the second by Proposition \ref{p.class}.
Substituting this in the right hand side of \eqref{e.hilbertcone1}, and then substituting the resulting
formula for $H(\cc,x)$ into \eqref{e.cone0}, yields the formula \eqref{e.thmhilbert1}for the Hilbert series.

We now turn to the multiplicity
formula, equation \eqref{e.thmhilbert2}. 
By \eqref{e.productcone},
\begin{equation} \label{e.cone1}
\mult(X,x) = \mult(V, \{ 0 \}) \mult (\cn, x) =  \mult (\cn, x) = \mult (\cc,x).
\end{equation}
Write $a =\mult (\cc,x)$.  We claim that $a =  \ev_{-v}(i_{V_1}^*([\cc]_{V_1}))$.
Assuming the claim, the remainder of the proof is similar to the $K$-theory case.  The analogue of
\eqref{e.hilbertcone1} is
\begin{equation} \label{e.multcone1}
\mult (\cc,x) = \ev_{-v}(i_{V_1}^*([\cc]_{V_1})) = \ev_{-v} \Big(\frac{i_{V_1}^*([\cc]_{V_1})}{c_{d_1}^T(V_1)} \Big) ,
\end{equation}
where the first equality holds by the claim, and the second holds because any $\ga \in \Phi(V_1)$ satisfies $\ga(-v) = 1$.  
The analogue of \eqref{e.hilbertcone2} is
\begin{equation} \label{e.multcone2}
\frac{i_{V_1}^*([\cc]_{V_1})}{c_{d_1}^T(V_1)} =  \frac{i_{V_1}^*([\cc]_{V_1}) c_{d-d_1}^T(V/V_1)}{c_d^T(V)}
= \frac{i_V^*([\cn]_V)}{c_d^T(V)} =  \frac{ i_{V' \times V}^*([X]_{V' \times V})}{c_d^T(V)}.
\end{equation}
Substituting  this in the right hand side of \eqref{e.multcone1}, and then substituting the resulting formula for
$\mult(\cc,x)$ into \eqref{e.cone1}, yields the formula \eqref{e.thmhilbert2} for the multiplicity.

It remains to prove the claim.  As indicated in \cite[Section 9]{IkNa:09}, using the relationships between
$K$-theory and Chow groups, and between Hilbert series and multiplicity, the formula for the multiplicity in the claim
can be deduced from the corresponding formula in $K$-theory (which is given in \cite[Prop.~2.2]{GrKr:15}).  
A more direct argument, using Lemma \ref{lem.equivariant}, is as follows.  The claim asserts
that $ \ev_{-v}(i_{V_1}^*([\cc]_{V_1})) = a$, where
the integer $a$ is defined by the equation
$[\PP(\cc)] = a h^k \cap [\PP(V_1)]$.
Since $i_{V_1}^*([\cc]_{V_1})$ is a homogeneous element of $S(\hat{T})$ of
 degree $k$, where $k$ is  the codimension of $\cc$ in $V_1$, the claim is equivalent to the assertion that 
$ \ev_{-rv} (i_{V_1}^*([\cc]_{V_1})) = a r^k$ for some $r \neq 0$.
Let $T' \cong \G_m$ be the subtorus of $T$ corresponding to the
cocharacter $rv$, where $r$ is
a nonzero integer such that
 $rv$ is an integral element of $\ft$.  Let $v' \in \ft_1$ be the element mapping
to $-rv \in \ft$, and define $\gl \in \hat{T'}$ by
$\gl(v') = 1$.  Let $\res$ denote the restriction from $T$-equivariant
Chow groups to $T'$-equivariant Chow groups.  Restriction commutes with pullback
to a point, and moreover, given $\xi \in A^*_T(pt)$, we have $\ev_{-rv}(\xi) =
\ev_{v'}(\res(\xi))$.  Therefore,
\begin{equation} \label{evaluate3}
 \ev_{-rv} (i_{V_1}^*([\cc]_{V_1}))  = \ev_{v'} \res \Big(  (i_{V_1}^*([\cc]_{V_1}) \Big)
 = \ev_{v'} (i_{V_1}^*([\cc]_{V_1, T'}),
\end{equation}
where the subscript $T'$ indicates $T'$-equivariant Chow groups.
Since $T'$ acts on $V_1$ with all weights equal to $r \gl$,  Lemma
\ref{lem.equivariant} implies that $[\cc]_{V_1, T'} = a r^k \gl^k \cap [V_1]_{T'}$.
Therefore, $ \ev_{v'} (i_{V_1}^*([\cc]_{V_1, T'}) =  \ev_{v'} (a r^k \gl^k) = a r^k$.
The result follows.
\end{proof}

\begin{Rem}  \label{rem.vanish}
If some $\ga \in \Phi(V)$ satisfies $\ga(v) = 0$, then $\ev_{v}(\gl_{-1}(V^*)) = 0$
and $\ev_{-v}(c_d^T(V)) = 0$.  Nevertheless, the evaluations in \eqref{e.thmhilbert1}
and \eqref{e.thmhilbert2} can be carried out.  Indeed, the proof of Theorem
\ref{thm.hilbert} shows that
\begin{equation}  \label{e.vanish1}
 \frac{i_x^*[\co_{X}]}{\gl_{-1}(V^*)} = \frac{i_{V_1}^* ([\co_{\cc}]_{V_1})}{\gl_{-1}(V_1^*)},
\end{equation}
since the left hand side is obtained from the right hand side by multiplying numerator
and denominator by the common factor $\gl_{-1}((V/V_1)^*) $.
Note that $\ev_v$ can be applied to the right hand side of
\eqref{e.vanish1} since each $\ga \in \Phi(V_1)$ satisfies $\ga(v)
\neq 0$ (in fact, $\ga(v) = -1$).  However, if one starts out with the
formula on the left side of \eqref{e.vanish1}, to perform the
evaluation $\ev_v$, it is necessary first to cancel the common factors
of $1 - e^{-\ga}$ with $\ga(v) = 0$ from the numerator and
denominator.  This might be nontrivial if one has a complicated
expression for $i^*[\co_X]$.  But in the case of Schubert varieties,
in \cite{GrKr:19}, we show that it is possible to explicitly perform
this cancellation and apply the formula.  Similar remarks apply in the
Chow group situation.
\end{Rem}

\begin{example} \label{example.generalcominuscule}
Let $T = \G_m$.  Let $\mu: T \to \G_m$ be defined by $e^{\mu}(t) = t$.
Identify $\Z$ with $\hat{T}$ by the map $n \mapsto n \mu$;
then $R(T)$ is the span of $e^{n \mu}$ for $n \in \Z$, and
$A^*_T(pt)$ is the polynomial ring $\F[\mu]$.
Suppose $T$ acts on $V = \F^3$ with all weights equal to $-1$ (that is, $-\mu$).
Define $v \in \ft$ by $\mu(v) = 1$.
Let $x_1, x_2, x_3$ be coordinates on $V$, let $s = x_1^2+
x_2^2 + x_3^2$, and let $X$ denote the zero-scheme of $s$.
The origin is a generalized cominuscule point of $X$, with $v$ 
and $V$ as above, and $V' = \{ 0 \}$.
Write $A = \F[V] = \F[x_1, x_2, x_3]$ and $B = A/\langle s \rangle
= \F[X]$.  The function $s$ is a weight vector for $T$ with weight
$2$ (that is, $2 \mu$).  Thus, $s \otimes 1$ can be viewed as a $T$-invariant regular
section of the bundle $V \times \F_{-2 \mu} \to V$, so by Lemma \ref{lem.zerosection},
$[X] = - 2 \mu \cap [V]$.  Hence
$i_V^*[X] = - 2 \mu$, so by \eqref{e.thmhilbert2},
$$
\mult(X,0) = \ev_{\-v}(i_V^*[X]) = - 2 \mu(-v) = 2.
$$
Similarly, we can compute the Hilbert series.  Lemma \ref{lem.zerosection} implies that
$[\co_X] = (1 - e^{2 \mu}) [\co_V]$.
Hence $i_V^*[\co_X] = (1 - e^{2 \mu})$.
Since $\gl_{-1}(V^*) =
(1 - e^{\mu})^3$, we have
$$
\frac{i_V^*[\co_X]}{\gl_{-1}(V^*)} = \frac{1+e^{\mu}}{(1 - e^{\mu})^2},
$$
and thus
$$
H(X,0) = \ev_v \Big( \frac{1+e^{\mu}}{(1 - e^{\mu})^2} \Big) = \frac{1+t}{(1-t)^2}.
$$
\end{example}

\section{Schubert varieties and slices} \label{s.Schubertslice}
\subsection{Background} \label{ss.Schubertbackground}
Let $G$ be a semisimple algebraic group, and $B \supset T$
a Borel subgroup and maximal torus of $G$.
Let $U$ be the unipotent
radical of $B$, so $B = TU$.  The Lie algebra of
an algebraic group will be denoted by the corresponding fraktur letter, so that
the Lie algebras of these groups are (respectively)
$\fg, \fb, \ft$ and $\fu$.  Let $B^- = TU^-$ denote the opposite
Borel subgroup to $B$.  
Let $\Phi$ denote
the set of roots of $\ft$ on $\fg$, and let $\Phi^+ = \Phi(\fu) \subset \ft^*$;
that is, $\Phi^+$ is the
positive system of roots chosen so that the root spaces of $\fu$
correspond to roots in $\Phi^+$.  
The Weyl group is $W = N_G(T)/T$, with
longest element $w_0$.  We will write
$x$ both for an element of $W$ and for a representative in $N_G(T)$.
 If $H$ is a subgroup of $G$ normalized by $T$, and $x \in W$,
we will write $H(x) = x H x^{-1}$; this is independent of the choice of a representative for $x$
in $N_G(T)$.  The Lie algebra of $H(x)$ will be denoted
$\fh(x) = (\Ad x) (\fh)$.  Note that if $U'$ is a unipotent subgroup of $G$ normalized
by $T$, then as a $T$-variety, $U'$ is isomorphic to its Lie algebra $\fu'$, as
can be deduced from the results in \cite[Ch.~8]{Spr:09}.

Let $P \supset B$ be a standard parabolic subgroup of $G$.  Let $P = LU_P$ be a Levi decomposition
where $L \supset T$, and let $P^- = LU^-_P$ be the opposite parabolic.
Write $U_L = U \cap L$ and $U^-_L = U^- \cap L$.   The product map
$U_P^- \times U_L^- \to U^-$ is an isomorphism (this follows from standard results about unipotent
groups, cf.~\cite[Ch.~8]{Spr:09}). 
Let $W_P$ denote the Weyl group of $L$.  Let
$\Phi_L$ denote the set of roots of $\ft$ in $\fl$, and
$\Phi^+_L = \Phi_L \cap \Phi^+$, $\Phi^-_L = - \Phi^+_L$.  Thus, $\Phi^+ = \Phi^+_L \cup \Phi(\fu_P)$.

In each left coset $wW_P$
of $W_P$ in $W$ there is a unique element $w_m$ of minimal length.
We write $W^P$ for the set of minimal length coset representatives.
We have $w \in W^P$ if and only if $w \Phi^+_L \subset \Phi^+$,
or equivalently,
$w \ga >0$ for all simple
roots in $\Phi^+_L$ (cf. \cite[Lemma 3.3]{Ric:92} and the discussion preceding
that lemma, as well as \cite[Remark 5.13]{Kos:61}).  
If $w \in W^P$, then $x \geq w$
$\Leftrightarrow$ $x_m \geq w$ (see \cite[Lemma 2.8]{BoGr:03}).  Observe that if $x,y \in W$ satisfy $xW_P = y W_P$, then $U_P(x) = U_P(y)$ and
$U_P^-(x) = U_P^-(y)$. 
If $x$ is minimal in its left $W_P$-coset,
then $U^-(x) \cap U = U_P^-(x) \cap U$ (cf.~\cite{Knu:09}).

\begin{Lem} \label{lem.rootsign}
Let $x \in W$.
\begin{myenumerate}
\item If $\ga \in \Phi^+$, then $x \ga >0$ if and only if
$x s_{\ga} > x$.
\item If $\ga \in x \Phi^-$, then $s_{\ga} x > x$ if and only if
$ \ga \in \Phi^-$.
\end{myenumerate}
\end{Lem}

\begin{proof} (1) follows from \cite[Lemmas 8.10, 8.11]{BGG:75}.  For (2), write
$\gb = - x^{-1} \ga \in \Phi^+$.  We have
$$
 s_{\ga} x =  x x^{-1} s_{\ga} x = x s_{x^{-1} \ga}  = x s_{-\gb} x = x s_{\gb} .
 $$
By (1), $x s_{\gb} > x$ if and only if $x \gb = - \ga \in \Phi^+$, that is, if
$\ga \in \Phi^-$.
\end{proof}

\begin{Def} \label{def.maxparabolic}
Let $w \in W$.  Suppose $P = L U_P$ is the standard parabolic subgroup such that the simple roots of 
$\Phi^+_L$ are those simple roots $\ga \in \Phi^+$
with $w \ga >0$.  We say $P$ is the standard
$X^w$-maximal parabolic determined
by $w$, and we denote it by $P^w$. (This parabolic subgroup is considered in \cite[Def.~1.1]{Per:07}.)
\end{Def}

Although $P^w$ is 
generally not a maximal parabolic in the usual sense, the following proposition shows
that it
is the largest parabolic subgroup such that $w \in W^P$.

\begin{Prop} \label{prop.maxparabolic}
Let $w \in W$.  If $P = P^w$, then $w \in W^P$.  Moreover,
if $Q$ is any standard parabolic subgroup, then $w \in W^Q$
if and only if $Q \subset P$.
\end{Prop}

\begin{proof} 
Suppose that $Q$ is a standard parabolic subgroup with
Levi factor $M$ containing $T$.
As noted above, $w \in W^Q$ if and only
if for each simple root $\ga$ in $\Phi^+_M$, we have 
$w \ga > 0$.
By definition,
the standard maximal parabolic subgroup $P = LU$ is defined
so that the simple roots in $\Phi^+_L$ are exactly the simple
roots $\ga$ with $w \ga >0$.  Hence $w \in W^P$.  Moreover, $w \in W^Q$
if and only if the simple roots in $\Phi^+_M$ are a subset
of the simple roots of $\Phi^+_L$, which occurs if and
only if $Q \subset P$.
\end{proof}

\subsection{Schubert varieties}
In this section we recall some basic facts about Schubert varieties; one reference for some of these
facts is \cite{BiLa:00}.  We include some
proofs for convenience or lack of a reference.

 Let $X = G/B$ denote the flag variety.  The $T$-fixed points in
 $X$ are the points $xB$, where $x \in W$ (as usual we abuse notation and write $x$ for either an element of
 $W$ or a representative in $N(T)$).  More generally, if $P = L U_P \supset B$ is a standard parabolic subgroup, we write
 $X_P$ for the generalized flag variety $G/P$.  The $T$-fixed points of
$X_P$ are of the form $x P$, for $x \in W$; write $i_{x,P}$ for
the inclusion of $xP$ into
$X_P = G/P$.  

Given
$w \in W$,
$X^w_{0,P} = B^- \cdot w P$ is a subvariety of $X_P$ isomorphic to affine space, called a Schubert cell.  Its
closure is the Schubert variety
$X^w_P = \overline{B^- \cdot w P}$.  
Since $xP = yP$ if and
only if the cosets $x W_P$ and $y W_P$ are equal,  $X^w_P$ depends only on the coset of $w$ mod $W_P$.  Write
$[\co_{X^w_P}]$ for the classes
of the structure sheaves in $K_T(G/P)$, and
$[X^w_P]$ for the fundamental classes
in $A^T_*(G/P)$.  If we are assuming $P=B$, then we generally omit the subscript $P$ from the notation, e.g., 
we write $X^w$, $i_x$, etc.

\begin{Prop} \label{prop.Schubertimage}
Let $P$ be a standard parabolic subgroup of $G$, and let $w \in W$.  Let $\pi: G/B \to G/P$ be
the projection.
\begin{myenumerate}
\item We have $\pi(X^w) = X^w_P$.
\item If $w \in W^P$, then $\pi^{-1}(X^w_P) = X^w$.  
Conversely, if $X^w$ is the inverse image of
a Schubert variety in $G/P$, then $w \in W^P$. 

\item If $X^w$ is the inverse image of
$X^w_Q$ under the map $G/B \to G/Q$, then $Q \subset P^w$, so $G/B \to G/Q \to G/P^w$, and the inverse image
of $X^w_{P^w}$ in $G/Q$ is $X^w_Q$.
\end{myenumerate}
\end{Prop}

\begin{proof}
(1) The map $\pi$ is $B^-$-equivariant, so
\begin{equation} \label{e.schubertimage}
\pi(X^w_0) = \pi(B^- \cdot wB) = B^- \cdot \pi (wB) = B^- \cdot wP = X^w_{0,P}.
\end{equation}
Since the map $\pi$ is proper, it is a closed map.  We have
$$
\pi(X^w) = \pi(\overline{X^w_0}) =  \overline{\pi(X^w_0)}= \overline{X^w_{0,P}} = X^w_P,
$$
where the second equality holds because $\pi$ is a closed map.  This proves (1).

(2) First suppose that $w \in W^P$.  Since $X^w_P = \pi(X^w)$, we have
$\pi^{-1}(X^w_P) \supseteq X^w$; we must prove the reverse inclusion.
It suffices to show that if $X^y \subset \pi^{-1}(X^w_P)$, then $y \geq w$.
The hypothesis $X^y \subset \pi^{-1}(X^w_P)$ implies that
$yP$ is in $X^w_P$.  Every $T$-fixed point in $X^w_P$ is the image
of a $T$-fixed point in $X^w$, so $y P = zP$ for some $z \geq w$.
Since $w$ is minimal, $z_m \geq w$ (see Section \ref{ss.Schubertbackground}); since
$y W_P = z W_P = z_m W_P$, we have $y \geq z_m$, so $y \geq w$.
Hence $\pi^{-1}(X^w_P) = X^w$, as
desired.   For the converse, suppose
that $X^w = \pi^{-1}(X^y_P)$.  We must show that $w$ is
minimal.  Since $X^y_P = X^{y_m}_P$, we may
assume $y = y_m$.  Then by what we have already proved, $\pi^{-1}(X^y_P) = X^y$,
so $X^y = X^w$.  Hence $y = w$, so $w$ is minimal, as desired.

(3) Let $P = P^w$.  If $X^w$ is the inverse image of $X^w_Q$, then by (2), $w \in W^Q$, so Proposition \ref{prop.maxparabolic}
implies that $Q \subset P$.  Consider the projections
$$
G/B \stackrel{f}{\rightarrow} G/Q \stackrel{g}{\rightarrow} G/P.
$$
Since $g^{-1}(X^w_P)$ is closed, $U^-$-stable, and irreducible, it is a Schubert variety
$X^u_Q$ for some $u \in W^Q$.  By
(2), we have
$f^{-1}(X^u_Q) = X^u$; but $f^{-1}(X^u_Q) = f^{-1}(g^{-1}(X^w_P)) = X^w$.
Hence $u = w$,  so $g^{-1}(X^w_P) = X^w_Q$, as desired.
\end{proof}

\subsection{Slices to Schubert varieties} \label{s.slice}
Let $x \in W$.  The map $U^-_P(x) \to C_{x,P} = U^-_P(x) \cdot xP$ embeds $U^-_P(x)$ as an open
subvariety of $G/P$.  This embedding is $T$-equivariant, where
$T$ acts by conjugation on $U^-_P(x)$, and by left multiplication on
$C_{x,P}$. We refer to $C_{x,P}$ as an open cell.
If $Q = MU_Q \subset P = L U_P$ are standard parabolic subgroups, then
\begin{equation} \label{e.celliso}
C_{x,Q} \cong U_Q^-(x)  \cong U_P^-(x) \times (U^-_Q \cap L)(x) \cong C_{x,P} \times (U^-_Q \cap L)(x),
\end{equation}
as follows from the Lie algebra decomposition 
$\fu^-_Q = \fu^-_P \oplus (\fu^-_Q \cap \fl)$.   If $w \in W^P$, then
under the identification \eqref{e.celliso}, we have
\begin{equation} \label{e.schubertcelliso1}
X^w_Q \cap C_{x,Q}  =  (X^w_P \cap C_{x,P}) \times (U^-_Q \cap L)(x). 
\end{equation}
Indeed, the projection
$\pi: G/Q \to G/P$ takes $C_{x,Q}$ to
$C_{x,P}$.  Let $\rho := \pi|_{C_{x,Q}}$.  
If $Z$ is any subscheme of $C_{x,P}$, then under the
identification given by \eqref{e.celliso},
$\rho^{-1}(Z) = Z \times (U^-_Q \cap L)(x)$.
Equation \eqref{e.schubertcelliso1} follows from this, since  $\rho^{-1}(X^w_P \cap C_{x,P}) =X^w_Q \cap C_{x,Q}$ for
$w \in W^P$.

\begin{Def} \label{d.slice}
Let $x \geq w$ be elements of $W$  and let $P = P^w$.
Define
$$
\cn^w_{x} =   [(U^-_P(x) \cap U) \cdot xP] \cap X^w_P \subset C_{x,P} \cap X^w_P .
$$
More generally, if $Q$ is a standard parabolic subgroup contained in $P$ (so $w \in W^Q$), define
$$
\cn^w_{x,Q} = [(U^-_P(x) \cap U) \cdot xQ] \cap X^w_Q \subset C_{x,Q} \cap X^w_Q.
$$
\end{Def}

Note that $\cn^w_{x,Q}$ is $T$-stable since it is the intersection of two $T$-stable subvarieties.  
Under the identification \eqref{e.schubertcelliso1}, we have
\begin{equation} \label{e.slicefullpartial}
\cn^w_{x,Q} = \cn^w_x \times \{ 1 \} .
\end{equation}
The reason is that under the identifications \eqref{e.celliso} and \eqref{e.schubertcelliso1}, we have
\begin{eqnarray*}
\cn^w_{x,Q} & = & [(U^-_P(x) \cap U) \times \{1 \}] \cap  [(X^w_P \cap C_{x,P}) \times (U^-_Q \cap L)(x)] \\
& = & [(U^-_P(x) \cap U) \cap (X^w_P \cap C_{x,P})] \times \{1 \} =  \cn^w_x \times \{ 1 \} .
\end{eqnarray*}
Hence we can view $\cn^w_x$ as a subvariety of any generalized flag variety $G/Q$ with $Q \subset P^w$.  In particular,
we can view $\cn^w_x$ as a subvariety of $G/B$.  Note also that if $xW_P = y W_P$, then since $xP
= yP$ in $G/P$, we have $\cn^w_x = \cn^w_y$.

\begin{Lem} \label{lem.slice}
Let $H$ be a linear algebraic group, and $X$ a scheme.  Let
$H$ act on the product $H \times X$ by left multiplication on the
first factor.  Any $H$-invariant closed subscheme $Z$ of $H \times X$
is of the form $Z = H \times Y$, where
$Y$ is the closed subscheme $(\{e \} \times X) \cap Z$ of
$\{e \} \times X$ (which we identify with $X$).
\end{Lem}

Results of this form are known (cf.~the proof of \cite[Prop.~1.3.5]{Bri:05}), so we omit the proof.  The lemma implies that the action map gives
an isomorphism $H \times Y \to Z$; we will refer to $Y$ as
a slice to $Z$.  The next proposition
 is analogous to \cite[Prop.~1.3.5]{Bri:05} (which concerns Kazhdan-Lusztig varieties), and
has a similar proof.

\begin{Prop} \label{prop.schubertslice} 
Let $x \geq w$ be in $W$.  Let $P = P^w=LU_P$.
The action map gives a $T$-equivariant isomorphism
\begin{equation} \label{e.action1}
(U^-_P(x) \cap U^-) \times \cn^w_x \to X^w_P \cap C_{x,P}.
\end{equation}
\end{Prop}

\begin{proof}
By the remarks preceding the proposition, the
map \eqref{e.action1} is $T$-equivariant.  We must
prove that this map is an isomorphism.  We can decompose
$\fu^-_P(x) = (\fu^-_P(x) \cap \fu^-) \oplus (\fu^-_P(x) \cap \fu)$,
and hence 
\begin{equation} \label{e.productmaps1}
(U^-_P(x) \cap U^-) \times (U^-_P(x) \cap U) \cong U_P^-(x) \cong C_{x,P}.
\end{equation}
Since $C_{x,P}$ (resp.~$X^w_P$) is stable under the left action of $U^-_P(x)$ (resp.~$U^-$),
$X_P^w \cap C_{x,P}$ is stable under the left action of
$U^-_P(x) \cap U^-$. 
Applying Lemma \ref{lem.slice}, with $H = U^-_P(x) \cap U^-$, $X = (U^-_P(x) \cap U) \cdot xP$,
and $Z = X^w_P \cap C_{x,P}$, we see that under the identification
\eqref{e.productmaps1}, the embedding
$X^w_P \cap C_{x,P} \subset C_{x,P}$ corresponds to the embedding
$$(U_P^-(x) \cap U^-) \times \cn^w_x 
\subset (U^-_P(x) \cap U^-) \times (U^-_P(x) \cap U).
$$
This proves the result.
\end{proof}

The previous proposition shows shows that $\cn^w_x$ is a slice to $X^w_P$, where $P = P^w$.  In fact,
$\cn^w_x$ plays the role of a slice at $xQ$ to $X^w_Q$ for any $Q \subset P$, since by 
 \eqref{e.schubertcelliso1}, 
 \begin{equation} \label{e.schubertsliceQ}
X^w_Q \cap C_{x,Q} \cong (U^-_P(x) \cap U^-) \times \cn^w_x \times (U^-_Q \cap L)(x).
\end{equation}

\begin{Prop} \label{prop.dimslice}
Suppose $x \geq w$.  Let $P = P^w = LU_P$. Then
$$
\dim \cn^w_x = \ell(x)  - \ell(w) - |x \Phi^-_L \cap \Phi^+|.
$$
In particular, if $x \in W^P$, then $\dim \cn^w_x = \ell(x)  - \ell(w)$.
\end{Prop}

\begin{proof} 
Equation \eqref{e.action1} implies that
$$
\dim \cn^w_x = \dim X^w_P - \dim (U^-_P(x) \cap U^-).
$$
It is well known that 
the codimension of $X^w_P$ in $X_P$ is $\ell(w)$, so 
the dimension of $X^w$ is $\dim X_P - \ell(w) = \dim U^-_P(x) - \ell(w)$.
Hence
$$
\dim \cn^w_x = \dim U^-_P(x)  - \dim (U^-_P(x) \cap U^-) - \ell(w) = \dim( U^-_P(x) \cap U) - \ell(w)
$$
But 
$$
\dim (U_P^-(x) \cap U) = |x \Phi^- \cap \Phi^+| - |x \Phi^-_L \cap \Phi^+| = \ell(x) -  |x \Phi^-_L \cap \Phi^+|.
$$
The result follows.
\end{proof}

\begin{Rem} \label{rem.kl}
Given $w \in W^P$, and $x \geq w$, there is a Kazhdan-Lusztig variety $\cn^{w,KL}_{x, P}$ at $xP$,
defined by taking the intersection of $X^w_P$ with an opposite Schubert cell.  (Kazhdan and Lusztig defined this in the case
where $P = B$, but following \cite[Section 7.3]{Knu:09}, we use the term Kazhdan-Lusztig variety in this
more general context.)  If $P = P^w$, then $\cn^w_{x,P} = \cn^w_x$ equals the Kazhdan-Lusztig
variety $\cn^{w,KL}_{x, P}$.  Since $\cn^w_{x,P}$ is the intersection of $X^w_P$ with 
part of an opposite
Schubert cell, it can be smaller than $\cn^{w,KL}_{x, P}$.
\end{Rem}

The discussion above shows that slices to Schubert varieties in generalized flag varieties are also
slices in the full flag variety.  Along the same lines,
the next proposition shows that computations of pullbacks and Hilbert series for generalized
flag varieties can be carried out using the full flag variety.

\begin{Prop} \label{prop.hilbertfullpartial}
Let $P$ be a standard parabolic subgroup, and suppose that $w \in W^P$ and
$x \geq w$.
\begin{myenumerate}
\item  We have
\begin{equation} \label{e.pullback} 
 i_{x,P}^* [\co_{X^w_P}] = i_x^* [\co_{X^w}] \ \ \mbox{and} \ \ 
i_{x,P}^*[X^w_P] = i_x^*[X^w].
\end{equation}

\item Let $d = \dim G/B - \dim G/P = |\Phi^+_L|$.  Then
$$
H(X^w_P, xP) = (1-t)^d H(X^w, xB) \ \ \mbox{and} \ \
\mult(X^w_P, xP) = \mult(X^w, xB).
$$
\end{myenumerate}
\end{Prop}

\begin{proof}
(1) The proof of this equation for $K$-theory is in \cite[Section 2.2]{GrKr:15};
the proof for Chow groups is almost the same.

(2) If $w \in W^P$ and
$x \geq w$, then \eqref{e.schubertcelliso1}
implies
$$
X^w \cap C_x \cong (X^w_P \cap C_{x,P}) \times U^-_L(x).
$$
The group $U^-_L(x)$ is isomorphic to affine space of
dimension $d$, so
$H(U^-_L(x), e) = \frac{1}{(1-t)^d}$,
and $\mult(U^-_L(x), e) =1$.  The result now follows
immediately from \eqref{e.productcone}.
\end{proof}

To simplify the notation, we will write simply
$H(X^w,x)$ for $H(X^w,xB)$ and $\mult(X^w,x)$ for $\mult(X^w,xB)$.

\subsection{Tangent spaces to slices}
\label{ss.tangentslice}
We now describe the set of weights $\Phi(T_{xB} \cn^w_x)$ (for $x \geq w$)
in terms of $\Phi(T_{xB} X^w)$.  The sets $\Phi(T_{xB} X^w)$ have been described
in classical types (see \cite{LaSe:84}, Lakshmibai \cite{Lak:95},
\cite{Lak2:00}, \cite{Lak:00}, \cite[Chapter 5]{BiLa:00}).  Thus, in classical
types, we can describe $\Phi(T_{xB} \cn^w_x)$, which is what we need to
determine whether $x$ is cominuscule in $X^w$. Further combinatorial
refinements appear in \cite{GrKr:19}.  

Observe that
$$
\Phi(T_{xB} X) = x \Phi^- \supset \Phi(T_{xB} X^w) \supset \Phi(T_{xB} \cn^w_x).
$$

\begin{Prop} \label{prop.tangentslice} Let $w \in W$, and let
$P = P^w$.
If $x \geq w$, then
\begin{equation} \label{e.tangentslice1}
\Phi(T_{xB} \cn^w_x) = \Phi(T_{xB} X^w) \setminus ((x \Phi^- \cap \Phi^-)
\sqcup (x \Phi^-_L \cap \Phi^+)).
\end{equation}
Hence, if $x \in W^P$,
\begin{equation} \label{e.tangentslice2}
\Phi(T_{xB} \cn^w_x) = \Phi(T_{xB} X^w) \setminus (x \Phi^- \cap \Phi^-).
\end{equation}
\end{Prop}

\begin{proof}
By Proposition \ref{prop.schubertslice}, we have
$$
\Phi(T_{xP} X^w_P) = \Phi(T_{xB} \cn^w_x) \sqcup (x \Phi(\fu^-_P) \cap \Phi^-).
$$
If $\pi: G/B \to G/P$ denotes the projection, then
\begin{eqnarray}
\Phi(T_{xB} X^w) & = &  \Phi(T_{xP} X^w_P) \sqcup \Phi(T_{xB}(\pi^{-1}(xP)) \\
& = & \Phi(T_{xP} X^w_P) \sqcup x \Phi^-_L.
\end{eqnarray}
Therefore,
\begin{equation} \label{e.tangentslice3}
\Phi(T_{xB} X^w) = \Phi(T_{xB} \cn^w_x) \sqcup (x \Phi(\fu^-_P) \cap \Phi^-)
\sqcup x \Phi^-_L.
\end{equation}
Since $\Phi^- = \Phi(\fu^-_P) \sqcup \Phi^-_L$, we can rewrite \eqref{e.tangentslice3}
as
$$
\Phi(T_{xB} X^w) = \Phi(T_{xB} \cn^w_x) \sqcup (x \Phi^- \cap \Phi^-)
\sqcup (x \Phi^-_L \cap \Phi^+).
$$
This implies \eqref{e.tangentslice1}. 
If $x \in W^P$, then $x \Phi^-_L \subset \Phi^-$,
so  \eqref{e.tangentslice2} follows.
\end{proof}

Suppose $x \geq w$.  In type $A$, a result of Lakshmibai and Seshadri (\cite{LaSe:84}) implies that 
\begin{equation} \label{e.tanSchA}
\Phi(T_{xB} X^w) = \{ \ga \in x \Phi^- \mid s_{\ga} x \geq w \}.
\end{equation}
This set contains the elements $\ga \in \Phi^-$ such that $s_{\ga} x > x$, which
by Lemma \ref{lem.rootsign}, equals $x \Phi^- \cap \Phi^-$.  Combining this
with Proposition \ref{prop.tangentslice} yields the following.

\begin{Prop} \label{prop.tangentsliceA} Suppose $G$ is of type $A$.  Let $w \in W$, and let
$P = P^w$.
If $x \geq w$, then
\begin{eqnarray*} 
\Phi(T_{xB} \cn^w_x) & = & \{ \ga \in x \Phi^- \setminus x \Phi^-_L \mid x > s_{\ga} x \geq w \} \\
& = & \{ \ga \in x \Phi^- \setminus (x \Phi^-_L \cap \Phi^+) \mid x > s_{\ga} x \geq w \} \\
& = & \{ \ga \in x \Phi(\fu_{P^w}^-) \mid x > s_{\ga} x \geq w \}.
\end{eqnarray*}
Hence, if $x \in W^P$,
\begin{equation*} 
\Phi(T_{xB} \cn^w_x) = \{ \ga \in x \Phi^- \mid x > s_{\ga} x > w \}.
\end{equation*}
\end{Prop}

Note that the first equality of the proposition holds because if $\ga \in x \Phi^-_L \cap \Phi^-$, then $s_{\ga} x > x$.

\begin{Rem}
For $G$ of arbitrary type, a result of Carrell and Peterson implies the right hand sides of the equations in
Proposition \ref{prop.tangentsliceA}  describe the sets of weights to
tangent spaces of $T$-stable curves in $\cn^w_x$ through $xB$.
\end{Rem}

%---------------------------------------

\section{Cominuscule points of Schubert varieties} \label{s.cominusculeSchubert}
The definition of generalized cominuscule points allows for arbitrary slices $\cn$.  For Schubert varieties, we define cominuscule points using the particular slices
$\cn^w_x$.  Since by \eqref{e.schubertsliceQ}, $\cn^w_x$ serves as a slice at $xQ$
to $X^w_Q$ 
for any standard parabolic
subgroup $Q$ such that $w \in W^Q$,
this definition can be used to calculate Hilbert series and multiplicities in $X_Q$ for any
such $Q$.  Since these Hilbert series and multiplicities can be determined
from the corresponding $X^w$ in the full flag variety $X$ (Proposition \ref{prop.hilbertfullpartial}), we will restrict
our attention to the case $Q = B$.

\subsection{Cominuscule points} \label{ss.cominuscule}
In this section we define the notion of a cominuscule point in a Schubert variety.
We provide some examples of cominuscule points, and in type $A$, we give conditions
which guarantee that a cominuscule point in a Schubert variety is also a cominuscule
point in a smaller Schubert variety.

\begin{Def} \label{def.standardcominuscule} Suppose $x \geq w$  are
  elements of $W$.  We will say that
  $x$ is cominuscule in $X^w$  if
  there exists $v \in \ft$ (which can be assumed to be rational) such that for all $\ga$ in $\Phi(T_{xB}
  \cn^w_x)$, we have $\ga(v)
  = -1$.  
\end{Def}

Note that we can assume that $v$ is rational by the discussion in Remark \ref{rem.rational}.

\begin{Rem} \label{rem.wporbit}
Let $P = P^w$, and suppose $x \geq w$.
If $x W_P = y W_P$, then since $\cn^w_x = \cn^w_y$, we see that if $x$ is cominuscule
in $X^w$ then so is $y$ for any $y \in x W_P$.  Moreover, in this case, 
if $Q \subset P$ is a standard parabolic subgroup of $G$ (so $w \in W^Q$),
then
$xQ$ is a generalized cominuscule point of $X^w_Q$.  
\end{Rem}

\begin{example} \label{ex.firstexamples}
For any $w \in W$, the element $w$ is cominuscule in $X^w$.  Indeed, the slice $\cn^w_w$ is a single point,
so the cominuscule condition is trivially satisfied.
For a similar reason, any $x$ is cominuscule in $X^1 = X$.  Indeed, in this case, 
since $w = 1$, we have $P^w = G$, so $U^-_{P^w} = \{ 1 \}$.  Thus for
any $x \in W$, $\cn^1_{x}$ is a point, so the cominuscule condition is trivially satisfied.
\end{example}

We now show that Schubert varieties in cominuscule flag
varieties give rise to points which are cominuscule in our sense.
We begin by recalling the definition of cominuscule flag varieties. 
Suppose that $P = LU$ is a maximal standard parabolic subgroup; then there is a unique simple root $\gb$ which is
not in $\Phi_L$.  The parabolic subgroup $P$ and the corresponding flag variety $G/P$ are said to be of cominuscule type
if the simple root $\gb$ appears with coefficient equal to $1$ when the highest root is written as a sum of simple roots.
In this case, it is known that for any $x \in W$ there exists an element $v \in \ft$ such that for any
$\ga \in \Phi(T_{xP}(G/P))$, $\ga(v) = -1$ (one reference is \cite[Prop.~2.9]{GrKr:15}).

\begin{Prop} \label{prop.cominflag}
Suppose $P$ is a standard parabolic subgroup of cominuscule
type.  If $x \in W^P$, then for any 
$w\leq x$, $x$ is cominuscule in $X^w$.
\end{Prop}

\begin{proof}
By Remark \ref{rem.wporbit}, we may assume $w \in W^P$.
The tangent space $T_{xP} (\cn^w_{x,P})$ 
is 
a subspace of $T_{xP} (G/P)$, so the cominuscule condition on the tangent space of the normal slice holds because it holds
for $T_{xP} (G/P)$.
\end{proof}

The next result gives conditions under which an element $x$ which is cominuscule
in $X^w$ is also cominuscule in a smaller Schubert variety $X^v$.  We do not know if
the result is valid in types other than type $A$, since the descriptions of tangent
spaces are more complicated in other types.

\begin{Cor} \label{cor.orderideal}
Let $G$ be of type $A$.  
Suppose that 
$w \leq v \leq x$ are elements of $W$ such that $P^w \subset P^v$
(equivalently, $v \in W^{P^w}$).
If $x$ is cominuscule in $X^w$, then $x$ is cominuscule in $X^v$.
\end{Cor}

\begin{proof}
Since $P^w \subset P^v$,  the reverse inclusion
$\Phi(\fu_{P^w}^-) \supset \Phi(\fu_{P^v}^-)$ holds.  Since also $w \leq v$, Proposition
\ref{prop.tangentsliceA} implies $\Phi(T_{xB} \cn^v_x) \subset \Phi(T_{xB} \cn^w_x)$,
implying the result.
\end{proof}

Without the condition $P^w \subset P^v$, the assumption that $x$ is cominuscule in $X^w$
does not imply that $x$ is cominuscule in $X^v$.
Indeed,
if this were true, then since any $T$-fixed point is cominuscule
in $X^1 = X$ by Example \ref{ex.firstexamples}, any $T$-fixed point would be cominuscule
in any Schubert variety, which is not true (see Example \ref{ex.noncomin_pt}).

\subsection{Cominuscule elements of Weyl groups} \label{ss.cominusculeweyl}
\begin{Def} \label{d.cominweyl}
The element $x \in W$ is cominuscule if and only if there 
 exists $v \in \ft$ such that for all $\ga \in x \Phi^- \cap \Phi^+ = I(x^{-1})$,
 we have $\ga(v) = -1$.
 \end{Def}

This notion is due to Peterson, with different terminology: the element $x$
is cominuscule if for some $\gl \in \ft^*$, $x$ is $\gl$-cominuscule (in Peterson's sense)
for the dual root system.  See \cite[Prop.~5.1]{Ste:01}.

It follows from the equality $I(x) = - x^{-1} I(x^{-1})$ that $x$ is cominuscule if and
only if $x^{-1}$ is.
In type $A_{n-1}$, the Weyl group is the symmetric group $S_n$,
and the cominuscule elements are exactly the $321$-avoiding
permutations (see \cite[p.~25]{Knu:09}).  

The next proposition shows that cominuscule elements provide cominuscule points, although
the examples of Section \ref{s.examples} show that
not all cominuscule points arise this way.

\begin{Prop} \label{prop.Petersoncomin}
If $x$ is a cominuscule element of $W$, then $x$ is cominuscule in any
Schubert variety $X^w$ containing $xB$ (equivalently, such that $x \geq w$). 
\end{Prop}

\begin{proof}
If $x$ is cominuscule, then
there exists $v \in \ft$ such that for all $\ga \in I(x^{-1})$,
we have $\ga(v) = -1$.  By Proposition \ref{prop.tangentslice}, 
for any $w \leq x$, we have
$\Phi(T_{xB} \cn^w_x) \subset
I(x^{-1})$, so by definition, $x$ is cominuscule in $X^w$.  
\end{proof}

\begin{Rem} \label{rem.comin}
Knutson observed that for a $\gl$-cominuscule element $x$ of
$W$, the torus action on the cell $X_x^0 = U \cdot xB$ contains the natural dilation action (see \cite[p.~25]{Knu:09}), and therefore
the ideal of $X^w \cap X_x^0$ in $X_x^0$ is homogeneous with
respect to the standard dilation action.  It was noted
in \cite{LiYo:12} this condition implies that the Kazhdan-Lusztig ideal can be used to compute Hilbert series and multiplicities.
In the context of this paper, 
Proposition \ref{prop.Petersoncomin} is an almost immediate
consequence of Knutson's observation.
\end{Rem}

\begin{Rem} \label{rem.cominconversefalse} The converse of Proposition
  \ref{prop.Petersoncomin} is false.  For example, in type $A_2$, the
  long element $w_0$ of $W$ is cominuscule in any $X^w$, but in
  $1$-line notation, $w_0 = (3,2,1)$, which is not cominuscule in $W$.
\end{Rem}

\subsection{Hilbert series and multiplicity formulas at cominuscule points} \label{ss.Hilbertcomin}
The following theorem is a straightforward consequence of the corresponding result
for generalized cominuscule points (Theorem \ref{thm.hilbert}).  

\begin{Thm} \label{thm.hilbertschubert} Let $N = \dim G/B = |\Phi^+|$.  Let $w \in W$.
Suppose $x$ is cominuscule
  in $X^w$.  Let $v$ be as in Definition \ref{def.standardcominuscule}.  Let $P = LU$ be the standard $X^w$-maximal parabolic subgroup, and let
  $d' = N - \ell(x) + |x \Phi^-_L \cap \Phi^+|$.  
 Then
  \begin{align}
    H(X^w,x) &   = \frac{1}{(1-t)^{d'} }
    \ev_v\Big(\frac{i_x^*[\co_{X^w}]}{\prod_{\ga \in x \Phi(\fu_P^-) \cap \Phi^+ }(1-e^{-\ga})}\Big) \label{e.hilbert1}\\
    \mult(X^w,x) & = \ev_{-v} \Big( \frac{i_x^*[X^w]}{\prod_{\ga\in
         x \Phi(\fu_P^-) \cap \Phi^+ }\ga}\Big). \label{e.mult1}
    \end{align}
    Note that if $x \in W^P$, then since $x \Phi^-_L \cap \Phi^+$ is empty, we have
    $d' = N - \ell(x)$ and $x \Phi(\fu_P^-) \cap \Phi^+ = x \Phi^- \cap \Phi^+ = I(x^{-1})$.
\end{Thm}

\begin{proof}
In light of the decomposition \eqref{e.schubertsliceQ} (with $Q=B$), this follows from Theorem \ref{thm.hilbert}, with 
$V' = (\fu^-_P(x) \cap \fu^-) \oplus \fu_L^-(x)$,
$V = \fu_P^-(x) \cap \fu$, and $\cn = \cn^w_x$.  We have
\begin{eqnarray*}
d' & = & \dim V' = | x \Phi(\fu_P^-) \cap \Phi^-| + |x \Phi_L^-| \\
& = & |x \Phi^- \cap \Phi^-| - |x \Phi^-_L \cap \Phi^-| + |x \Phi_L^-| \\
& = & |x \Phi^- \cap \Phi^-| + |x \Phi^-_L \cap \Phi^+| = N - \ell(x) + |x \Phi^-_L \cap \Phi^+|,
\end{eqnarray*}
which is the formula for $d'$ in the statement of the theorem.
\end{proof}

\begin{Rem} \label{rem.vanishschubert}
As noted in Remark \ref{rem.vanish}, the cancellations necessary to perform the evaluations
are explicitly described in \cite{GrKr:19}, where in fact it is shown
that the formulas can be evaluated without having to find $v$.
\end{Rem}

As a corollary to this theorem, if $x$ is a cominuscule element of
$W$, we obtain formulas  closely resembling the formulas for cominuscule flag
varieties given in \cite[Prop.~12]{IkNa:09} and \cite[Theorem 2.10]{GrKr:15}.

\begin{Cor} \label{thm.hilbertschubertW}
Let $N = \dim G/B$.
Suppose that $x$ is a cominuscule element of $W$ and $x \geq w$.
Let $v$ be an element of $\ft$ such that $\ga(v) = -1$ for all $\ga \in x \Phi^- \cap \Phi^+$.  Then
$$
H(X^w,x)  = \frac{\ev_v(i_x^*[\co_{X^w}])}{(1-t)^N} \ \ \mbox{and} \ \ 
\mult(X^w,x)  = \ev_{-v}(i_x^*[X^w]). 
$$
\end{Cor}

\begin{proof}
These formulas follow immediately from Theorem \ref{thm.hilbertschubert} because for all $\ga \in x \Phi(\fu_P^-) \cap \Phi^+$, we
have $\ga(v) = -1$.
\end{proof}

\section{Examples}\label{s.examples}

In this section we apply the results of earlier sections to Schubert
varieties in $G/B$, where $G$ is of type $A_5$. We describe the ingredients
of the computations, but omit most details.  The Hilbert series and multiplicities are computed
using formulas $i_x^*[X^w]$ (due
Anderson-Jantzen-Soergel and Billey (\cite{AJS:94}, \cite{Bil:99}) and
$i_x^*[\co_{X^w}]$ (due to Graham and Willems
(\cite{Gra:02}, \cite{Wil:06}).  These formulas can be found (in a version consistent
with the conventions of this paper) in \cite{GrKr:15}.  We will not restate the
formulas here, but note that they depend on the choice of reduced expression for
$x$ and are related to the number of subexpressions multiplying to $w$.

In type $A_{n-1}$, $G=\SL_n$. Let $B$ be the Borel subgroup of upper triangular matrices in
$G$, and $T$ the maximal torus of diagonal matrices in $G$. The set of
positive roots of $G$ is $\Phi^+=\{\gre_i-\gre_j\mid 1\leq
i < j\leq n, i\neq j\}$.  The simple roots are $\ga_1, \ldots, \ga_{n-1}$
with $\ga_i = \gre_i-\gre_{i+1}$.  
The Weyl group $W$ is the permutation group $S_n$,
with $w\in W$ acting on $\gre_i-\gre_j$ by
$w(\gre_i-\gre_j)=\gre_{w(i)}-\gre_{w(j)}$. The $i$-th simple reflection is the
transposition
$s_i=(i,i+1)$.  

We will make use of the following descriptions of tangent spaces (for $x \geq w$):
$$
\Phi(T_{xB} X^w) = \{ \ga \in x \Phi^- \mid s_{\ga} x \geq w \},
$$
by  \eqref{e.tanSchA}, and
and if $x \in W^{P^w}$,
$$
\Phi(T_{xB} \cn^w_x) = \{ \ga \in x \Phi^- \mid x > s_{\ga} x > w \},
$$
by Proposition \ref{prop.tangentsliceA}.  To use these descriptions of tangent spaces,
we need the following characterization of the Bruhat order in type $A$ (see \cite[Section 5.9]{Hum:90}).  If $x = (x(1), \ldots, x(n))$
and $y = (y(1), \ldots, y(n))$ are two permutations in $S_n$, written in $1$-line
notation, then $x \leq y$ in the Bruhat order if and only if for each $i \in \{1, \ldots, n \}$,
the following holds: if $a_1, a_2, \ldots a_i$ are the numbers $x(1), x(2), \ldots, x(i)$
written in increasing order, and $b_1, b_2, \ldots b_i$ are the numbers $y(1), y(2), \ldots, y(i)$
written in increasing order, then $a_k \leq b_k$ for all $k \in \{1, \ldots, i \}$.

\begin{example}\label{ex.noncomin_pt}
Let $w = (3,4,1,6,2,5)$ in 1-line notation.
  We have $\dim X = |\Phi^+| = 15$, and $\dim X^w = \dim X - \ell(w) 
  = 15 - 6 = 9$.  Let $P = P^w = LU_P$.  The simple roots in
$\Phi^+_L$ are $\{ \gre_1 - \gre_2, \gre_3-\gre_4, \gre_5 - \gre_6
\}$.
We show that for $y = (5,6,2,4,1,3) \in W^P$, $yB$ is not a cominuscule point of $X^w$. 
We have
  \begin{equation*}
    \begin{split}
      \Phi(T_{yB} \cn^w_y) &=\{\gre_1-\gre_2,\gre_3-\gre_4,\gre_3-\gre_5,
      \gre_4-\gre_5,\gre_4-\gre_6 \}\\
      \Phi(T_{yB} X^w_y)&=\Phi(T_{yB} \cn^w_y)\cup \{-(\gre_1-\gre_3),-(\gre_2-\gre_3),
     -(\gre_2-\gre_4),-(\gre_5-\gre_6 ) \}
    \end{split}
  \end{equation*}
  Because $\Phi(T_{yB} \cn^w_x)$ contains the roots $\gb_1 = \gre_3 - \gre_4$, $\gb_2 =
  \gre_4 - \gre_5$ and $\gb_1+\gb_2 = \gre_3 - \gre_5$, $yB$ is not a
  cominuscule point of $X^w$, since 
  if  $\gb_1(v) = \gb_2(v) = -1$ then 
 $(\gb_1+\gb_2)(v) = -2$.  Note that 
 $yB$ is a nonsingular point of $X^w$ since
 $\dim T_{yB} X^w = 9 = \dim X^w$, so
 $\mult(X^w,yB) = 1$.
\end{example}

\begin{example}\label{ex.comin_pttwo}
Let $w=(4,3,1,6,2,5)$ and let $x=(5,6,3,4,1,2)$.
We have $\dim X^w=\dim X - l(w) = 15-7 = 8$.
 Let $P = P^w  = LU_P$; then $\Phi^+_L=\{\gre_3-\gre_4,\gre_5-\gre_6\}$.  
 We have $x \in W^P$, and
  \begin{align*}
   \Phi(T_{xB} \cn^w_x) &= \{\gre_1-\gre_3, \gre_1-\gre_4, \gre_2-\gre_3,
    \gre_2-\gre_4,
    \gre_3-\gre_6, \gre_4-\gre_5,\gre_4-\gre_6\}\\
    \Phi(T_{xB} X^w_x)&=\Phi(T_{xB} \cn^w_x) \cup \{-(\gre_1-\gre_2), -(\gre_3-\gre_4),
    -(\gre_5-\gre_6) \}
  \end{align*}
  Since $\dim T_{xB}X^w=10>8=\dim X^w$, $X^w$ is singular at
  $x$. If $v=\frac{1}{2}\diag(-1,-1,-1,1,1,1)$, then
  $\ga(v)=-1$ for all $\ga\in \Phi(T_{xB}\cn_x^w)$, so $xB$ is a
  cominuscule point of $X^w$. The calculation of multiplicity and Hilbert
  series is facilitated by choosing a reduced expression for $x$ which has
  few subexpressions multiplying to $w$.  Motivated by the combinatorial
  results of \cite{GrKr:19} (where the calculations are carried out using pipe dreams),
we choose reduced expressions for $w$ and $x$ given by 
  (using the shorthand $i_1 i_2 \cdots i_k$ for $s_{i_1} s_{i_2} \cdots s_{i_k}$)
  $$
w = 3213254 , \ \ x = 432154324354.
$$
There are only $3$ subexpressions
of $x$ multiplying to $w$.  One verifies that 
the multiplicity of $X^w$ at $x$ is $3$.  Going further, one can verify that the Hilbert series
is given by
\begin{equation*}
  H(X^w,x)=\frac{-3(t-1)^7(t+1)^4-2(t-1)^8(t+1)^4}{-(t-1)^{15}(t+1)^4}
  =\frac{3}{(t-1)^8}+\frac{2}{(t-1)^7},
\end{equation*}
where the middle expression is obtained from Theorem \ref{thm.hilbertschubert} before simplifying.
We can recover the multiplicity from the Hilbert series by observing
that
\begin{equation*}
  \frac{3}{(t-1)^8}+\frac{2}{(t-1)^7} 
  = \sum_{k=0}^{\infty}\left(3\binom{k+7}{7}- 2\binom{k+6}{6}\right)t^k.
\end{equation*}
Thus the Hilbert polynomial $h(X^w,x)(k)=3\binom{k+7}{7}-
2\binom{k+6}{6}$. The leading term of this polynomial is
$\frac{3}{7!}k^7$, so $\mult(X^w,x)=3$.
 \end{example}

 \begin{example} \label{ex.secondexample} By a calculation similar to the
 previous example, we can verify that $x=(5,6,3,4,1,2)$ (the same $x$ as that
 example) is also cominuscule in
   $X^w$ for $w = (3,4,1,6,2,5)$. If we take the reduced expression for $x$ as
   in the previous example, then there are again only $3$ subexpressions multiplying to $w$, and
   the multiplicity of $X^w$ at $x$ is again $3$.  Note that if we took
  the reduced expression $x =  214354213254$, there would be $15$ subexpressions of $x$ multiplying
  to $w$, so the calculation would be more complicated.
  \end{example}
  
\begin{example} \label{ex.orderidealexamples}  
   Taking $x$ and $w$ as in the previous example, we can produce other Schubert varieties
   in which $x$ is cominuscule using Corollary
   \ref{cor.orderideal}, which states that $x$ is also cominuscule in
   $X^v$, provided that $w \leq v \leq x$ and $P^w \subset P^v$.
   For
   example, if we take $v = (3,5,1,6,2,4)$ then this holds (with $P^w =
   P^v$); if $v = (3,4,5,6,1,2)$ it holds (with $P^w \subsetneqq
   P^v$).
\end{example}

\begin{example}
Taking $x$ and $w$ as in the previous example, whenever $y$ is in the
   coset $xW_{P^w}$
   Remark \ref{rem.wporbit} implies that $y$ is cominuscule in
   $X^w$, and moreover that the multiplicity and Hilbert series of
   $X^w$ at $y$ are the same as those at $x$. Noting that $W_{P^w}=\langle s_1,s_3,s_5\rangle$, we
   could take, for example, $y=(5,6,4,3,1,2)$ or
   $y=(6,5,4,3,2,1)=w_0$.
\end{example}

\bibliographystyle{amsalpha}
\bibliography{stabilizernotes}

\def\cprime{$'$}
\providecommand{\bysame}{\leavevmode\hbox to3em{\hrulefill}\thinspace}
\providecommand{\MR}{\relax\ifhmode\unskip\space\fi MR }
\providecommand{\MRhref}[2]{%
  \href{http://www.ams.org/mathscinet-getitem?mr=#1}{#2}
}
\providecommand{\href}[2]{#2}
\begin{thebibliography}{BGG75}

\bibitem[AJS94]{AJS:94}
H.~H. Andersen, J.~C. Jantzen, and W.~Soergel, \emph{Representations of quantum
  groups at a {$p$}th root of unity and of semisimple groups in characteristic
  {$p$}: independence of {$p$}}, Ast\'erisque (1994), no.~220, 321.

\bibitem[BB93]{BeBi:93}
Nantel Bergeron and Sara Billey, \emph{R{C}-graphs and {S}chubert polynomials},
  Experiment. Math. \textbf{2} (1993), no.~4, 257--269.

\bibitem[BG03]{BoGr:03}
Brian~D. Boe and William Graham, \emph{A lookup conjecture for rational
  smoothness}, Amer. J. Math. \textbf{125} (2003), no.~2, 317--356.

\bibitem[BGG75]{BGG:75}
I.~N. Bern{\v{s}}te{\u\i}n, I.~M. Gel{\cprime}fand, and S.~I. Gel{\cprime}fand,
  \emph{Differential operators on the base affine space and a study of
  {${\fg}$}-modules}, Lie groups and their representations ({P}roc. {S}ummer
  {S}chool, {B}olyai {J}\'anos {M}ath. {S}oc., {B}udapest, 1971), Halsted, New
  York, 1975, pp.~21--64.

\bibitem[Bil99]{Bil:99}
Sara~C. Billey, \emph{Kostant polynomials and the cohomology ring for {$G/B$}},
  Duke Math. J. \textbf{96} (1999), no.~1, 205--224.

\bibitem[BL00]{BiLa:00}
Sara Billey and V.~Lakshmibai, \emph{Singular loci of {S}chubert varieties},
  Progress in Mathematics, vol. 182, Birkh\"auser Boston Inc., Boston, MA,
  2000.

\bibitem[Bri99]{Bri:99}
M.~Brion, \emph{Rational smoothness and fixed points of torus actions},
  Transform. Groups \textbf{4} (1999), no.~2-3, 127--156, Dedicated to the
  memory of Claude Chevalley.

\bibitem[Bri05]{Bri:05}
Michel Brion, \emph{Lectures on the geometry of flag varieties}, Topics in
  cohomological studies of algebraic varieties, Trends Math., Birkh\"auser,
  Basel, 2005, pp.~33--85.

\bibitem[EG98]{EdGr:98}
Dan Edidin and William Graham, \emph{Equivariant intersection theory}, Invent.
  Math. \textbf{131} (1998), no.~3, 595--634.

\bibitem[EG05]{EdGr:05}
\bysame, \emph{Nonabelian localization in equivariant {$K$}-theory and
  {R}iemann-{R}och for quotients}, Adv. Math. \textbf{198} (2005), no.~2,
  547--582.

\bibitem[Eis95]{Eis:95}
David Eisenbud, \emph{Commutative algebra}, Graduate Texts in Mathematics, vol.
  150, Springer-Verlag, New York, 1995, With a view toward algebraic geometry.

\bibitem[FK96]{FoKi:96}
Sergey Fomin and Anatol~N. Kirillov, \emph{The {Y}ang-{B}axter equation,
  symmetric functions, and {S}chubert polynomials}, Proceedings of the 5th
  {C}onference on {F}ormal {P}ower {S}eries and {A}lgebraic {C}ombinatorics
  ({F}lorence, 1993), vol. 153, 1996, pp.~123--143.

\bibitem[Ful84]{Ful:84}
William Fulton, \emph{Intersection theory}, Ergebnisse der Mathematik und ihrer
  Grenzgebiete (3) [Results in Mathematics and Related Areas (3)], vol.~2,
  Springer-Verlag, Berlin, 1984.

\bibitem[GK08]{GrKu:08}
William Graham and Shrawan Kumar, \emph{On positivity in {$T$}-equivariant
  {$K$}-theory of flag varieties}, Int. Math. Res. Not. IMRN (2008), Art. ID
  rnn 093, 43.

\bibitem[GK15]{GrKr:15}
William Graham and Victor Kreiman, \emph{Excited {Y}oung diagrams, equivariant
  {$K$}-theory, and {S}chubert varieties}, Trans. Amer. Math. Soc. \textbf{367}
  (2015), no.~9, 6597--6645.

\bibitem[GK19]{GrKr:19}
\bysame, \emph{The combinatorics of cominuscule points in {S}chubert
  varieties}, work in progress, 2019.

\bibitem[Gra02]{Gra:02}
William Graham, \emph{Equivariant ${K}$-theory and {S}chubert varieties},
  preprint (2002).

\bibitem[Hum90]{Hum:90}
James~E. Humphreys, \emph{Reflection groups and {C}oxeter groups}, Cambridge
  Studies in Advanced Mathematics, vol.~29, Cambridge University Press,
  Cambridge, 1990.

\bibitem[Ike16]{Ike:16}
Takeshi Ikeda, \emph{Lectures on equivariant {S}chubert polynomials}, Adv.
  Stud. Pure Math. \textbf{71} (2016), 97--137, Schubert calculus--Osaka 2012.

\bibitem[IN09]{IkNa:09}
Takeshi Ikeda and Hiroshi Naruse, \emph{Excited {Y}oung diagrams and
  equivariant {S}chubert calculus}, Trans. Amer. Math. Soc. \textbf{361}
  (2009), no.~10, 5193--5221.

\bibitem[KM05]{KnMi:05}
Allen Knutson and Ezra Miller, \emph{Gr\"obner geometry of {S}chubert
  polynomials}, Ann. of Math. (2) \textbf{161} (2005), no.~3, 1245--1318.

\bibitem[Knu09]{Knu:09}
Allen Knutson, \emph{Frobenius splitting, point-counting, and degeneration},
  2009.

\bibitem[Kos61]{Kos:61}
Bertram Kostant, \emph{Lie algebra cohomology and the generalized
  {B}orel-{W}eil theorem}, Ann. of Math. (2) \textbf{74} (1961), 329--387.

\bibitem[Kre05]{Kre:05}
Victor Kreiman, \emph{{S}chubert classes in the equivariant {K}-theory and
  equivariant cohomology of the {G}rassmannian}, ar{X}iv:math.{A}{G}/0512204,
  2005.

\bibitem[Kre06]{Kre:06}
\bysame, \emph{{S}chubert classes in the equivariant k-theory and equivariant
  cohomology of the {L}agrangian {G}rassmannian}, ar{X}iv:math.{A}{G}/0602245,
  2006.

\bibitem[Lak95]{Lak:95}
V.~Lakshmibai, \emph{Tangent spaces to {S}chubert varieties}, Math. Res. Lett.
  \textbf{2} (1995), no.~4, 473--477.

\bibitem[Lak00a]{Lak2:00}
\bysame, \emph{On tangent spaces to {S}chubert varieties}, J. Algebra
  \textbf{230} (2000), no.~1, 222--244.

\bibitem[Lak00b]{Lak:00}
\bysame, \emph{On tangent spaces to {S}chubert varieties. {II}}, J. Algebra
  \textbf{224} (2000), no.~2, 167--197.

\bibitem[LS84]{LaSe:84}
V.~Lakshmibai and C.~S. Seshadri, \emph{Singular locus of a {S}chubert
  variety}, Bull. Amer. Math. Soc. (N.S.) \textbf{11} (1984), no.~2, 363--366.

\bibitem[LY12]{LiYo:12}
Li~Li and Alexander Yong, \emph{Some degenerations of {K}azhdan-{L}usztig
  ideals and multiplicities of {S}chubert varieties}, Adv. Math. \textbf{229}
  (2012), no.~1, 633--667.

\bibitem[Per07]{Per:07}
Nicolas Perrin, \emph{Small resolutions of minuscule {S}chubert varieties},
  Compos. Math. \textbf{143} (2007), no.~5, 1255--1312.

\bibitem[Ric92]{Ric:92}
R.~W. Richardson, \emph{Intersections of double cosets in algebraic groups},
  Indag. Math. (N.S.) \textbf{3} (1992), no.~1, 69--77.

\bibitem[Spr09]{Spr:09}
T.~A. Springer, \emph{Linear algebraic groups}, second ed., Modern Birkh\"auser
  Classics, Birkh\"auser Boston, Inc., Boston, MA, 2009.

\bibitem[Ste01]{Ste:01}
John~R. Stembridge, \emph{Minuscule elements of {W}eyl groups}, J. Algebra
  \textbf{235} (2001), no.~2, 722--743.

\bibitem[Tho87]{Tho:87}
R.~W. Thomason, \emph{Algebraic {$K$}-theory of group scheme actions},
  Algebraic topology and algebraic {$K$}-theory ({P}rinceton, {N}.{J}., 1983),
  Ann. of Math. Stud., vol. 113, Princeton Univ. Press, Princeton, NJ, 1987,
  pp.~539--563.

\bibitem[Wil06]{Wil:06}
Matthieu Willems, \emph{{$K$}-th\'eorie \'equivariante des tours de {B}ott.
  {A}pplication \`a la structure multiplicative de la {$K$}-th\'eorie
  \'equivariante des vari\'et\'es de drapeaux}, Duke Math. J. \textbf{132}
  (2006), no.~2, 271--309.

\end{thebibliography}

\end{document}